\documentclass[11pt,reqno]{amsart}
\usepackage{graphicx}
\usepackage{amsmath}
\usepackage{amssymb}
\usepackage{amsthm}
\usepackage{amsfonts}

\usepackage{enumerate}
\usepackage[scriptsize,hang,raggedright]{subfigure}
\usepackage[cmyk]{xcolor}
\usepackage{color}

\numberwithin{equation}{section}

\newtheorem{theorem}{Theorem}[section]
\newtheorem{proposition}[theorem]{Proposition}
\newtheorem{lemma}[theorem]{Lemma}

\newtheorem{definition}[theorem]{Definition}
\theoremstyle{remark}
\newtheorem{remark}[theorem]{Remark}

\newcommand{\bthm}{\begin{theorem}}
\newcommand{\ethm}{\end{theorem}}

\newcommand{\bprop}{\begin{proposition}}
\newcommand{\eprop}{\end{proposition}}

\newcommand\eps{\varepsilon}
\newcommand\exd{\mathrm{d}}
\newcommand\dive{\mathrm{div}}

\newcommand\R{\mathbb{R}}
\newcommand\vect{\mathcal{X}}
\newcommand\BV{{BV}}

\newcommand\kp{\kappa}

\newcommand{\domainn}{{\mathbb{T}^n}}

\newcommand{\Hone}{\mathcal{H}^1}
\newcommand{\Htwo}{\mathcal{H}^2}

\newcommand{\beqn}{\begin{equation}}
\newcommand{\eeqn}{\end{equation}}

\newcommand{\pr}{\prime}
\newcommand{\pt}{\partial}
\newcommand{\arrow}{\rightarrow}

\newcommand{\rthree}{{\mathbb{R}^3}}

\newcommand{\domainS}{{\mathbb{S}^2}}

\renewcommand{\leq}{\leqslant}
\renewcommand{\geq}{\geqslant}

\DeclareMathOperator*{\argmin}{arg\,min}

\title[A nonlocal isoperimetric problem on $\domainS$]{Axisymmetric critical points of a nonlocal isoperimetric problem on the two-sphere}

\author{Rustum Choksi}
\author{Ihsan Topaloglu}
\author{Gantumur Tsogtgerel}
\address{Department of Mathematics and Statistics, McGill University, 
Montr\'{e}al, Qu\'{e}bec, Canada}
\email{rchoksi@math.mcgill.ca}
\email{itopaloglu@math.mcgill.ca}
\email{gantumur@math.mcgill.ca}
\date{\today}                                           
\begin{document}
\begin{abstract}
On the two dimensional sphere, we consider axisymmetric critical points of an isoperimetric problem perturbed by a long-range interaction term. 
When the parameter controlling the nonlocal term is sufficiently large, 
we  prove the existence of a local minimizer with arbitrary many interfaces in the axisymmetric class of admissible functions.  These local minimizers in this restricted class are shown to be critical points in the broader sense (i.e., with respect to \emph{all} perturbations). We then explore the rigidity, due to curvature effects,  in the criticality condition via several quantitative results regarding the axisymmetric critical points. 
\end{abstract}
%
%
%
\maketitle

\section{Introduction}

In this article we consider the energy functional
 		\beqn\label{e:nlip-energy}
			E_\gamma(u) = \frac12 \int_\domainS |\nabla u| + \gamma \int_\domainS |\nabla v|^2\,\exd\Htwo,
		\eeqn
over
	  \[
			\BV(\domainS;\{\pm1\}) = \left\{u\in\BV(\domainS)\, : \, \Htwo \big(\{x\in \domainS\colon u(x)\not\in\pm1\}\big)\, =\, 0\right\}
 		\] 
subject to the mass constraint
		\[
			\frac1{4\pi}\int_\domainS u\,\exd \Htwo = m.
		\]
Here $\gamma>0$ is a fixed parameter, and $v$ is a solution of the Poisson problem
		\beqn\label{e:poisson-v}
			-\Delta v = u - m,
		\eeqn
where $-\Delta$ denotes the Laplace--Beltrami operator on $\domainS$. Also, throughout this paper, unless noted otherwise, $\nabla$ denotes the gradient relative to $\domainS$.

Note that the first term of the energy is $1/2$ times the {\em total variation} of $u$, and, since $u$ takes on only values $\pm1$, it calculates the perimeter of the set $\{x\in\domainS\colon u(x)=1\}$.

The functional $E_\gamma$ arises, up to a constant factor, as the
$\Gamma$-limit as $\epsilon\to 0$ of the well-studied Ohta--Kawasaki
sequence of functionals $E_{\epsilon,\gamma}$ which model microphase
separation of diblock copolymers at the diffuse level, \cite{OK}:
		\begin{equation}\label{nlpch}
				E_{\epsilon, \gamma}(u) :=
					\begin{cases}
 						\int_{\domainS}   \frac{\epsilon}{2} |\nabla u|^2+\frac{(1-u^2)^2}{4\epsilon}+\gamma\,|\nabla v|^2 \,\exd\Htwo
 									&\;\, \text{if } u \in H^1(\domainS)\\
 									&\;\,\text { and }   \frac{1}{4\pi}\int_{\domainS} u \,\exd\Htwo = m, \\ \\
									+ \infty & \;\, \text{otherwise},
					\end{cases}
   	\end{equation}
where again $v$ satisfies \eqref{e:poisson-v}.		
		
Pattern formation of ordered structures on curved surfaces arises in systems ranging from biology to material science: e.g. covering virus and radiolaria architecture, colloid encapsulation for possible drug delivery, etc. (cf. \cite{E,Ga,KKG,VAB}). From the point of view of diblock copolymers, the self-assembly in thin melt films confined to the surface of a sphere was investigated in \cite{GC} computationally by looking at a model using the self-consistent mean field theory. In \cite{Tang} the authors look at the patterns emerging as a result of phase separation of diblock copolymers numerically on spherical surfaces by using the Ohta--Kawasaki model.

From the point of view of mathematical analysis, previous work on surfaces involves only the local energies, that is, \eqref{nlpch} and \eqref{e:nlip-energy} with $\gamma=0$. The authors of \cite{GLR} look at the local energy $E_{\epsilon, 0}$ and consider the effect of the Gauss curvature of the domain. On the other hand, for the sharp interface version $E_0$ it was shown that the global minimizer of the classical isoperimetric problem on $\domainS$ is the single cap, i.e., the set with boundary consisting of a single circle (cf. \cite{HHM,R}). Also recently, in \cite{BDF}, the authors established the stability of the isoperimetric domains on $\domainS$ by proving a quantitative version of the isoperimetric inequality on the sphere.

 Extensive mathematical analysis of the Ohta--Kawasaki model \eqref{nlpch} or its sharp interface limit  \eqref{e:nlip-energy} has been carried out on both the flat-tori and bounded domains in the Euclidean space (cf. \cite{ACO,AFM,CP,CP2,CiSp,GMS1,GMS2,Mu,MS,PV,RW2,RW3,RW8,Sp,ST} and references therein). However, the analysis on a curved surface is rare. To our knowledge the only rigorous analysis of \eqref{e:nlip-energy} defined on the two-sphere is carried out in \cite{Top}. There the author establishes the regularity of local minimizers of \eqref{e:nlip-energy} and characterizes the global minimizer of $E_\gamma$ in the small non-locality parameter regime. Indeed, for $\gamma>0$ small enough the global minimizer of $E_\gamma$ agrees with the global minimizer of the local isoperimetric problem posed on the two-sphere, namely it is the single spherical cap for any mass constraint $m\in(-1,1)$. Moreover, by looking at the second variation of $E_\gamma$, a stability analysis is presented for the single cap and double cap critical points. This analysis relies on the fact that we have an explicit formula of the Green's function for the Laplace--Beltrami operator on $\domainS$ (see Section \ref{sec:prelim}). 

\bigskip
In this article we address  the nonlocal problem posed on $\domainS$, i.e., the minimization of \eqref{e:nlip-energy} over $BV(\domainS;\{\pm 1\})$ subject to a mass constraint.
To this end, we focus on  axisymmetric critical points $E_\gamma$. These patterns are described by functions
	\[
		u \in BV(\domainS;\{\pm 1\}),\quad\mbox{with}\quad u=u(\phi)
	\]
in standard spherical coordinates $(r,\theta,\phi)$ where $\phi$ denotes the angle between the radius vector and the $z$-axis (see Figure \ref{examplesofaxisymm} for examples of such patterns and Section \ref{sec:prelim} for a precise definition).

	\begin{figure}[ht!]
\includegraphics[width=0.27\textwidth]{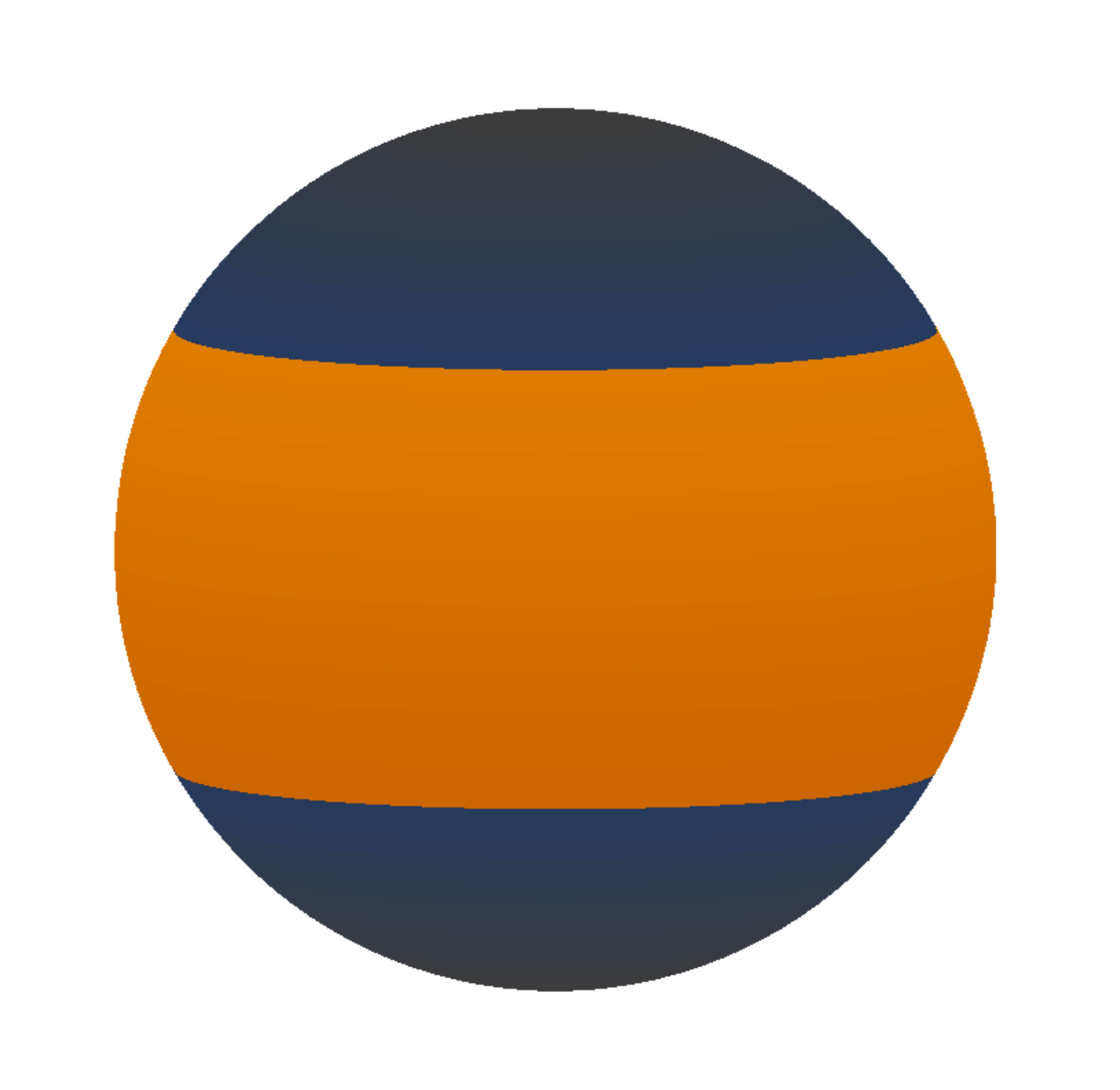}\quad\includegraphics[width=0.27\textwidth]{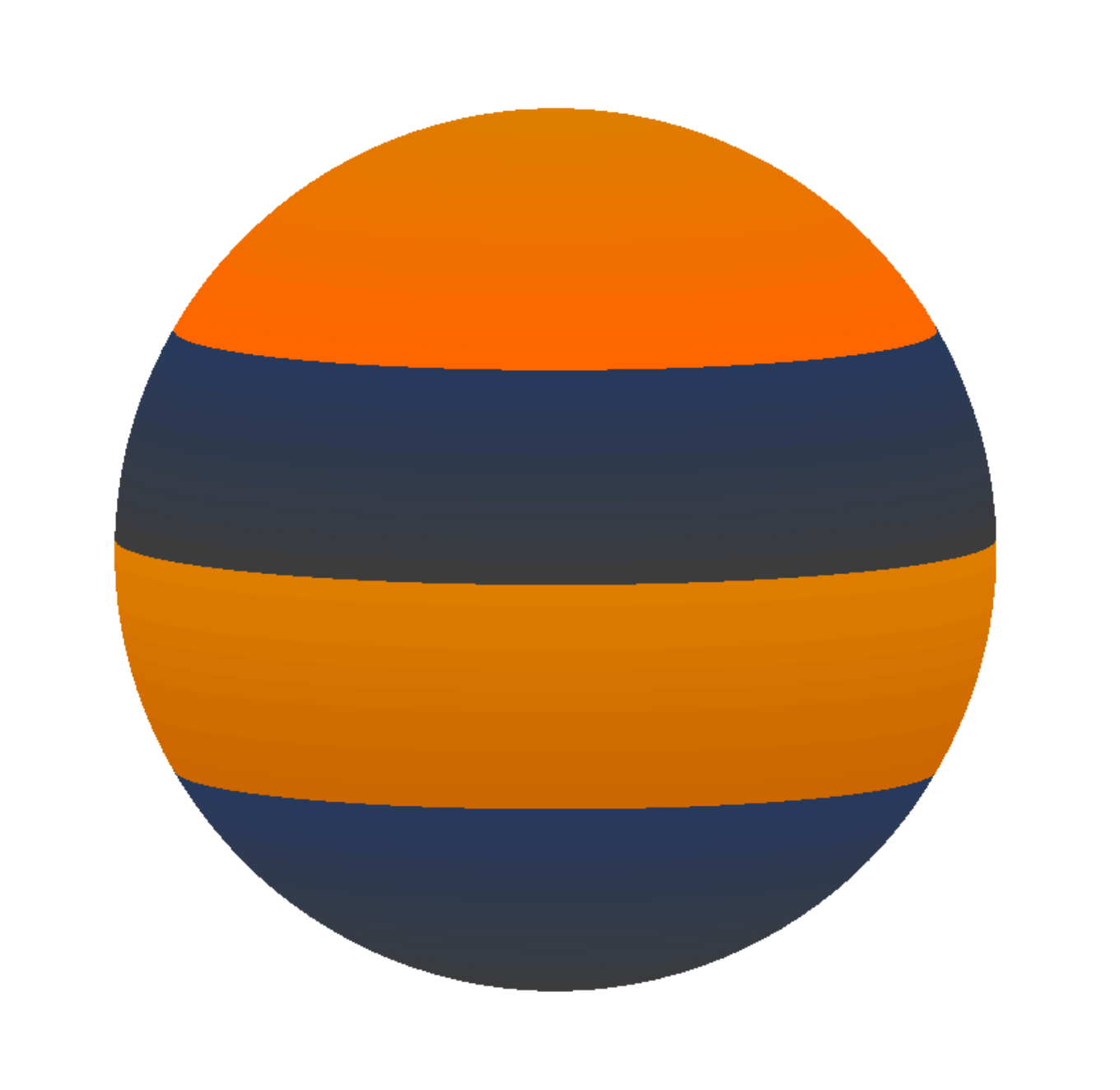}\quad\includegraphics[width=0.27\textwidth]{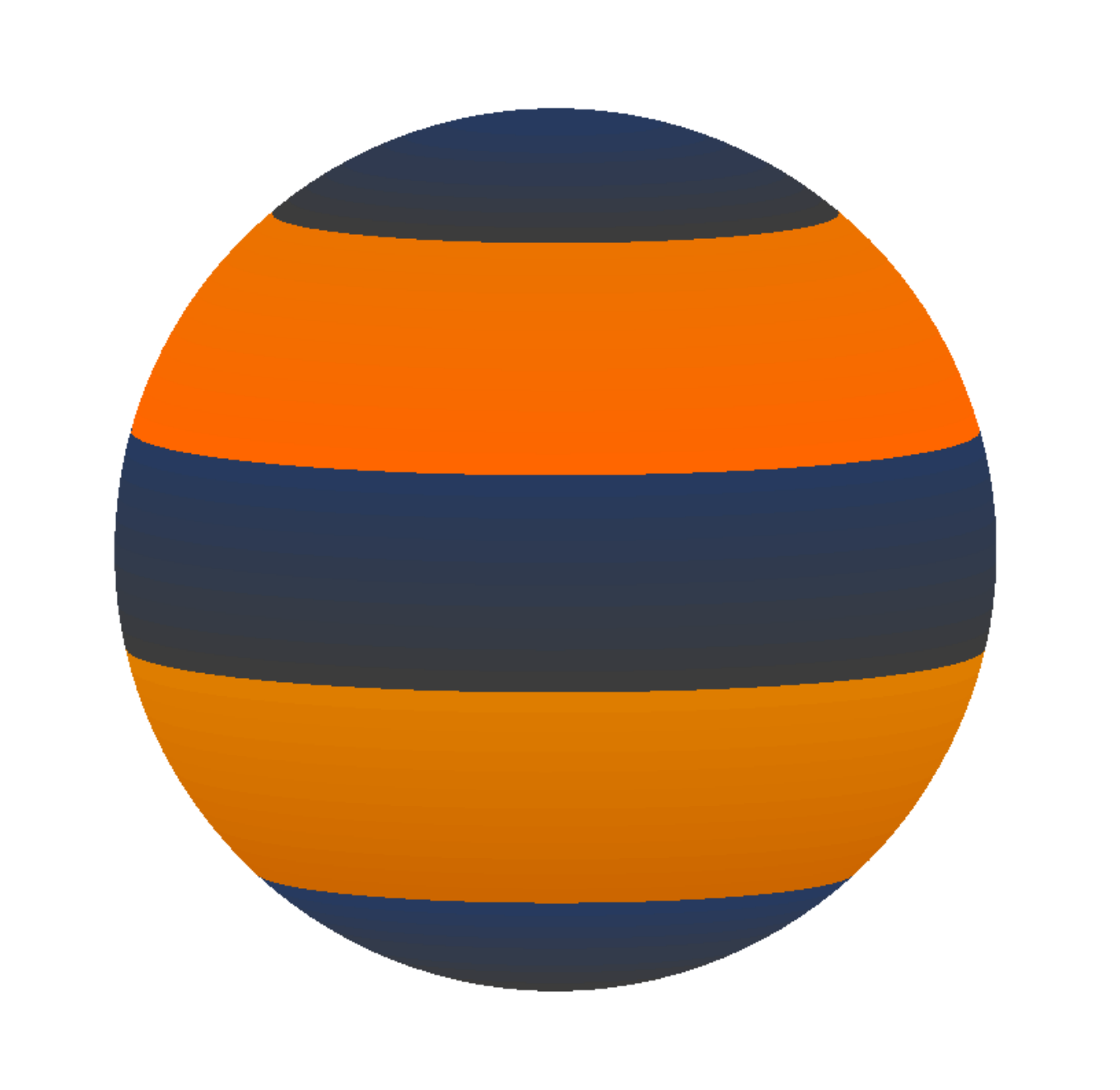}
  \caption{Examples of axisymmetric patterns on the sphere with 2, 3 and 4 interfaces}
  \label{examplesofaxisymm}
 	\end{figure}
 	
Taking axisymmetry as an ansatz allows us to write the energy in a one dimensional form (see Section \ref{sec:axisymm}) and this enables us to understand the effect of the local and nonlocal terms explicitly. Indeed, one might assume that if we restrict our admissible class to the axisymmetric functions with a finite number of interfaces the problem becomes straightforward. However, this is not the case due to the nonlinear and oscillatory nature of the nonlocal term. On the other hand, this one dimensional 
ansatz essentially turns the nonlocal term into a local contribution, allowing us to prove  the following in Theorem \ref{thm:existence}: For any fixed $n\in\mathbb{N}$, if $\gamma$ is sufficiently large enough, $E_\gamma$ admits an axisymmetric local minimizer with $n$ interfaces in the class of axisymmetric competitors.

We establish this result by considering a sequence of \emph{elementary moves}, i.e., perturbations in the $z$-variable restricted to three consecutive interfaces, and by looking at the change in the contribution to the energy by the perimeter and nonlocal terms as a result of these elementary moves. Here, the local minimality is \emph{only} with respect to one dimensional perturbations in the $z$-variable. However, in Proposition \ref{p:critequiv} we easily show that such a  local minimizer in this restricted class is not only a critical point with respect to one dimensional perturbations but also a critical point with respect to \emph{all} perturbations, i.e., it is a solution of the general Euler--Lagrange equation \eqref{firstvar}.
 
The axisymmetric patterns on the sphere can be considered as analogs of lamellar patterns on the flat torus as they depend only on one variable, that is, the polar angle $\phi$: However, our analysis will show that the study of axisymmetric patterns on the sphere is quite different and far richer than that of lamellar patterns on the flat domains. This is, of course, due to the curvature effects of the domain. Indeed the distribution of the components of an axisymmetric critical point is tied to the curvature of the ambient domain. These effects are in particular evident when one looks at the criticality condition. Unlike the interfaces of lamellar patterns on the flat torus whose boundary components all have zero mean curvature, though it is equal to a constant on every component, the geodesic curvature on the interfaces of axisymmetric patterns yields a different constant on each component of $\pt A$ depending on $\phi$. Because of this fact the criticality condition \eqref{firstvar} is very rigid. Exploiting this rigidity we will investigate the axisymmetric critical points of $E_\gamma$. Axisymmetric class of critical points are {\it reasonable} candidates to investigate because of inherent symmetries of the problem. However the criticality condition must be considered among \emph{all} patterns. Indeed, we do not mean to suggest that the {\it all  stable} critical points of $E_\gamma$ are  axisymmetric,  and numerical evidence certainly suggests otherwise (cf. Figure \ref{fig:spiraling} of Section \ref{sec:instability}). We will mostly work in the regime $m=0$ as for axisymmetric patterns on the sphere this is the more interesting regime. Indeed, for fixed $\gamma>0$ if $m$ is sufficiently close to either $-1$ or $1$ the first term dominates the second one and the results in \cite{Top} can be modified to show that the single spherical cap, that is, the axisymmetric pattern with one interface, is the global minimizer of $E_\gamma$.

The role of the parameter $\gamma$ is crucial when considering critical points of $E_\gamma$. Unlike the lamellar patterns on flat surfaces, here, the criticality of the axisymmetric patterns depends on the magnitude of the non-locality. Also, a priori, for any fixed $\gamma>0$ one would expect uniformly distributed axisymmetric patterns to be critical points (see Definition \ref{defn:uniform} for the definition of uniform distribution); however, due to the curvature effects, this is not the case for $E_\gamma$ on $\domainS$. In particular, in this article, we prove that;  
	\begin{itemize}
		\item given any $\gamma>0$ an axisymmetric critical point $u$ cannot have arbitrarily many interfaces, that is, the number of interfaces that $u$ has is bounded from above (Proposition \ref{p:numberofinterfaces}).
		\item the only axisymmetric critical points of $E_\gamma$ for an interval of $\gamma$-values are the single cap and the symmetric double cap (Proposition \ref{p:intervalofgamma}).
		\item any uniformly distributed pattern with the number of interfaces greater than 4 is not a critical point (Proposition \ref{p:uniforminterface}).
		\item for any axisymmetric critical point of $E_\gamma$ with $n\geq 2$ interfaces where $z_1$ denotes the first interface the diameter of the polar cap determined by $z_1$ is bounded from below by a constant multiple of $1/\gamma$ (Proposition \ref{p:lowerbound}).
		\item In general, the distance between two interfaces, $|z_{k+1}-z_k|$ is bounded below by a constant multiple of ${\max\{|z_k|,|z_{k+1}|\}}/\gamma$ (Remark \ref{r:lowerbound}).
	\end{itemize}

These results provide only a partial picture about the axisymmetric critical points. An important question which remains open is about the distribution of the interfaces of an axisymmetric critical point for large values of $\gamma>0$. For given $\gamma>0$, suppose $u$ is an axisymmetric critical point with $n$ interfaces where $n$ is the greatest possible integer determined by Proposition \ref{p:numberofinterfaces}. Even though the nonlocal term prefers a periodic distribution of the interfaces in the $z$-variable, thanks to Proposition \ref{p:uniforminterface}, we know that this is not the case. Then the natural question pertains to the distribution of interfaces, in particular where they are more dense. Since, due to the axisymmetry ansatz, the sphere is almost flat in small neighborhoods around $z=0$ we conjecture that for any large $\gamma>0$ the interfaces of the corresponding stable axisymmetric critical point are periodic in $z$ not with respect to the Lebesgue measure $dz$ but with respect to a weighted measure $d\mu=\frac{1}{1-z^2}dz$, i.e., they accumulate around the equator of $\domainS$.

It is, of course, the next natural question to ask whether the axisymmetric patterns are local minimizers with respect to \emph{all} perturbations. As a first step in this direction, a stability argument can be used where the second variation of $E_\gamma$ about any critical point with respect to \emph{any} perturbation $f$ is given in \eqref{secondvar}. However, as noted in Section \ref{sec:instability}, simple stability arguments are not helpful here as we lack the knowledge of the exact location of the interfaces of an axisymmetric critical point.

\section{Preliminaries and Notation}\label{sec:prelim}

Let $\Hone$ and $\Htwo$ denote, respectively,  one and two dimensional Hausdorff measure. 
As noted in the introduction, the first term of the energy $E_\gamma$ is defined using the total variation of $u$. For a smooth Riemannian manifold $\mathcal{M}$, a function $u\in L^1_{\mathrm{loc}}(\mathcal{M})$ is said to be of {\em bounded variation} if the following quantity, called the {\em total variation} of $u$, is finite:
		\[
			\int_{\mathcal{M}} |\nabla u| := \sup_{\phi\in \vect_c(\mathcal{M})} \frac{\langle u,\dive\,\phi\rangle}{\|\phi\|_\infty},
		\]
where $\vect_c(\mathcal{M})$ is the set of all compactly supported smooth vector fields on $\mathcal{M}$. The space of functions of bounded variations is denoted by $\BV(\mathcal{M})$.
We are interested in the subset of $\BV(\mathcal{M})$, consisting of functions taking only the values $\pm1$,
which can be defined by
		\[
			\BV(\mathcal{M},\pm1) = \left\{u\in\BV(\mathcal{M})\, :\, \Htwo \big( \{x\in \mathcal{M}:u(x)\not\in\pm1\}\big) \, =\, 0\right\}.
		\]
	The nonlocal term involving $v$ in \eqref{e:nlip-energy} can be written explicitly using the Green's function $G=G(x,y)$ associated with the Poisson problem \eqref{e:poisson-v}. For each $x\in\domainS$, $G(x,y)$ satisfies
  \[
   -\Delta_y G(x,y)=\delta_x - \frac{1}{4\pi}\,\,\text{ on }\domainS,\,\,\,\qquad\int_\domainS G(x,y)\,\exd\Htwo_x=0,
  \]
where $\delta_x$ is a delta-mass measure supported at $x$, and, in particular, one can show, by writing out the Laplace--Beltrami operator in spherical coordinates explicitly, that for $x,y\in\domainS$
  \beqn
   G(x,y)=-\frac{1}{2\pi}\log|x-y|,
  \nonumber
  \eeqn
where $|\cdot|$ denotes the Euclidean norm, that is, $|x-y|$ is the chordal distance between $x$ and $y$ in $\mathbb{R}^3$. The functions $G$ and $v$ are then related by
  \beqn
   v(x)=-\frac{1}{2\pi}\int_\domainS \log(|x-y|)u(y)\,\exd\Htwo_y.
  \nonumber
  \eeqn
		
Next we recall from \cite{CS2,Top} the first and second variations of $E_\gamma$. Denoting by $A:=\{x\in\domainS\colon u(x)=1\}$, we see that if $u$ is a critical point of $E_\gamma$ such that $\pt A$ is $C^2$, then we have
		\beqn
			\kappa_g(x)+4\gamma\,v(x)=\lambda\,\,\,\,\text{for all } x\in\pt A,
		\label{firstvar}
		\eeqn
where $\lambda$ is a constant and $\kappa_g$ denotes the signed geodesic curvature of $\pt A$ with respect to the outer normal of to $A$ (cf. \cite{dC}). Moreover, the second variation of $E_\gamma$ about the critical point $u$ is given by
		\beqn
			\begin{aligned}
				J(f):=& \int_{\pt A} |\nabla_{\pt A} f|^2-(1+\kappa_g^2)f^2\,\exd\mathcal{H}^1_x \\
				      &\qquad + 8\gamma\,\int_{\pt A}\int_{\pt A}\left(-\frac{1}{2\pi}\log(|x-y|_{\rthree})\right)f(x)f(y)\,\exd\mathcal{H}^1_x \exd\mathcal{H}^1_y \\
				      &\qquad\quad\quad + 4\gamma\,\int_{\pt A}(\nabla v\cdot\nu)f^2\,\exd\mathcal{H}^1_x,
			\end{aligned}
		\label{secondvar}
		\eeqn
where $f$ is any smooth function on $\pt A$ satisfying the condition
 		\[
 		  \int_{\pt A} f(x)\,\exd\mathcal{H}^1_x=0.
 		\]
Here $\nabla_{\pt A} f$ denotes the gradient of $f$ relative to the manifold $\pt A$ and $\nu$ denotes the unit tangent of $\domainS$ which is normal to $\pt A$ pointing out of $A$ (see \cite{CS2,Top} for details).

\begin{remark}[Scaling of the radius]
Here we have fixed the radius of the domain sphere. Here we remark that by scaling, this choice is without loss of generality, i.e., one can choose \emph{either} $\gamma$ as the free parameter and consider the problem on the unit sphere $\domainS$, or fix $\gamma=1$, and let the radius of the sphere vary.
To this end,  consider the energy on a sphere of radius $R$, denoted by $\mathbb{S}^2_R$, and let the radius $R$ be the free parameter in the energy
	\beqn
			E_{R}(\tilde{u}) = \frac{1}{2} \int_{\mathbb{S}^2_R} |\nabla \tilde{u}|+ \int_{\mathbb{S}^2_R} |\nabla \tilde{v}|^2\,\exd\Htwo.
		\nonumber
	\eeqn 
Here $\tilde{u}\in BV(\mathbb{S}^2_R;\{\pm 1\})$ and satisfies the mass constraint
	\[
		\frac{1}{4\pi R^2} \int_{\mathbb{S}^2_R} u\,\exd\Htwo=m.
	\]
However, by looking at the scaling
	\[
		u(x):=\tilde{u}(x/R) \in BV(\domainS;\{\pm 1\})
	\]
we see that
	\beqn
		\begin{aligned}
		E_{R}(\tilde{u}) &= \frac{R}{2}\int_{\domainS}|\nabla u| + R^4 \int_{\domainS} |\nabla v|^2\,\exd\Htwo \\
														&= R E_{R^3}(u);
		\end{aligned}
		\nonumber
	\eeqn
hence, considering $E_R$ on $\mathbb{S}^2_R$ is equivalent to considering the energy $E_\gamma$ on $\domainS$ with $\gamma=R^3$.
\end{remark}

\bigskip

 Let us now give a precise definition of an axisymmetric pattern on $\domainS$ and fix some notation. A function $u \in BV(\domainS;\{\pm1\})$ which depends only on the polar angle $\phi$ in standard spherical coordinates can be expressed as a function $u:[-1,1]\to \{\pm1\}$ by considering the change of coordinates
	\[
		z=\cos\phi.
	\]
Although $u$ is originally defined on the sphere, here with an abuse of notation we use $u$ also to denote the associated function of the one variable $z$. 
\begin{definition}\label{defn:uniform}
For $-1=z_0<z_1<\ldots<z_n<z_{n+1}=1$, the piecewise constant function
	\beqn
	  u(z) = \begin{cases} -1 &\text{if}\quad -1\leq z < z_1 \\
\phantom{-}1 & \text{if}\quad\ z_1 \leq z < z_2 \\
\hfill\vdots\hfill & \hfill\vdots\hfill \\
(-1)^{n+1} & \text{if}\quad z_n\leq z \leq 1 \end{cases} 
		\label{e:axisymmfunc}
	\eeqn
is called an \emph{axisymmetric pattern} on $\domainS$ with $n$ interfaces located at $z_k$ for $k=1,\ldots,n$.

An axisymmetric pattern $u$ is said to be \emph{uniformly distributed} with respect to the $z$-coordinate if the distance between each interface is a fixed constant, i.e., $z_{k+1}-z_k=C$ for all $k=1,\ldots,n-1$. 
\end{definition}

Due to the polar symmetries of the sphere, we have $E_\gamma(-u)=E_\gamma(u)$ for any axisymmetric pattern $u$; and, $E_\gamma(\tilde{u})=E_\gamma(u)$ for $\tilde{u}(z):=u(-z)$.
For a function $u$ of the form \eqref{e:axisymmfunc}, the mass constraint becomes
		\beqn
			\frac1{4\pi}\int_\domainS u\,\exd\Htwo = \frac12\sum_{k=1}^{n+1}(-1)^k(z_k-z_{k-1}) = m.
			\label{e:massconstraint}
		\eeqn

In spherical coordinates the Laplacian on $\domainS$ is given by
		\[
			\Delta = \frac{\partial^2}{\partial\phi^2} + \cot\phi\frac{\partial}{\partial\phi}.
		\]
By the change of coordinates $z=\cos\phi$, we obtain 
		\[
			\Delta = \frac\partial{\partial z} (1-z^2) \frac\partial{\partial z}.
		\]
Given that $u\in\BV(\domainS)$ depends only on $z$, the problem \eqref{e:poisson-v} can be solved by repeated integration.
In fact, one integration will suffice since we only need $\nabla v$ in the energy.
Since
		\[
			\int_{\domainS} u \,\exd\Htwo = 2\pi \int_{-1}^1 u(z)\,\exd z,
		\]
$u$ is absolutely integrable on $[-1,1]$.
Therefore the function
		\beqn
			\xi(z) = \xi(-1) + \int_{-1}^z(u(t)-m)\,\exd t,
			\label{e:xi}
		\eeqn
is absolutely continuous on $[-1,1]$, and $\partial_z\xi(z)=u(z)$ for a.e. $z\in (-1,1)$. This 
means that
		\[
			v(z) = \int_{0}^z\frac{\xi(t)}{1-t^2}\,\exd t
		\]
satisfies
		\[
			\partial_z ((1-z^2)\partial_z v) = u-m,
		\]
almost everywhere. 
By uniqueness, all other solutions of \eqref{e:poisson-v} are obtained by adding constants to $v$.
Turning to the nonlocal term in the energy, we have
		\[
			\int_{\domainS} |\nabla v|^2\,\exd\Htwo = 2\pi \int_{-1}^1 \big(\sqrt{1-z^2} |\partial_z v(z)|\big)^2 \,\exd z = 2\pi \int_{-1}^1 \frac{\xi(z)^2}{1-z^2}\,\exd z.
		\] 
Note again the abuse of notation in using  $v$ (and $u$) both as a function on $\domainS$ and as a function on  $z \in [-1,1]$.

Since $u$ as a function of $z$ is bounded,  elliptic regularity gives $v$ as a function of $z$ in $ W^{2,p}$ for any $p<\infty$. Thus by the Sobolev embedding theorem, we have  $v\in C^{1,\alpha}$ for some $\alpha\in(0,1)$.

For an axisymmetric critical point,  the Euler--Lagrange equation \eqref{firstvar} becomes
	\beqn
		\kappa_g(z_k)+4\gamma v(z_k)=\lambda
		\label{e:firstvar1D}
	\eeqn
at each interface $z_k$ where $u$ changes sign, and it holds with the same constant $\lambda$ on the right-hand side which is determined by the mass constraint and depends only on $u$ and $\gamma$. The geodesic curvature at an interface, $\kappa_g(z_k)$, is given by
	\beqn
		(\kappa_g(z_k))^2=\frac{z_k^2}{1-z_k^2},
		\label{e:geod-curv}
	\eeqn
where the sign of $\kappa_g(z_k)$ depends on the unit normal vector on the interface $\{z=z_k\}$ tangent to $\domainS$ pointing outward from the sets where $u=1$.

 The condition \eqref{e:firstvar1D} is still a general criticality condition in the sense that it captures the criticality of an axisymmetric pattern with respect to every perturbation and not only with respect to axisymmetric perturbations.

\begin{remark}
By reducing our problem to 1D and using the triangular wave functions $\xi$, the nonlocal contribution of $E_\gamma$ essentially   localizes. 
We will exploit now this fact, in particular in the proof of the Theorem \ref{thm:existence} where we establish the existence of a local minimizer in the axisymmetric setting. 
A similar connection is associated with the relationship between the nonlocal Ohta-Kawasaki functional  and the local functional studied in \cite{Mul} (and later in \cite{AM,Yip}), and this connection allows one to prove periodicity results for minimizers. 
Even though we can formulate the nonlocal term as the $L^2$-norm of $\xi$ with respect to the weighted measure $dz/(1-z^2)$, we cannot express the perimeter term as a total variation of a 1D function with respect to this weighted measure.  Hence at this point, we are not able to adopt the arguments used in \cite{Mul} to obtain a periodicity result in the $z$-variable.
\end{remark}

\section{Energy for the axisymmetric case}\label{sec:axisymm}

We first express the energy $E_\gamma$ of an axisymmetric pattern in terms of the locations of the interfaces. 

\bprop[Energy of an axisymmetric pattern]\label{p:energy1D}
 Let $u$ be an axisymmetric function on the sphere with $n$ interfaces located at $-1<z_1<\ldots<z_n<1$ given by \eqref{e:axisymmfunc} satisfying the mass constraint \eqref{e:massconstraint} for any $m\in(-1,1)$. Then the energy $E_\gamma$ is given by
 	\beqn
		\begin{split}
			\frac1\pi E_\gamma(u)
						&=  - 4\gamma(1- m^2) + 2\sum_{k=1}^{n}\sqrt{1-z_k^2}\\
						&\quad+ \gamma \sum_{k=0}^{n}(\xi_k-((-1)^{k}+m)(1-z_k))^2\log\frac{1-z_k}{1-z_{k+1}} \\
						&\quad+ \gamma \sum_{k=0}^{n}(\xi_k+((-1)^{k}+m)(1+z_k))^2\log\frac{1+z_{k+1}}{1+z_{k}}.
		\end{split}
		\label{e:1Denergy}
	\eeqn
In particular, for $m=0$ we have
	\beqn
		\begin{split}
				\frac1\pi E_\gamma(u)
					&=  - 4\gamma + 2\sum_{k=1}^{n}\sqrt{1-z_k^2}
+ \gamma \sum_{k=0}^{n}(\xi_k-(-1)^{k}(1-z_k))^2\log\frac{1-z_k}{1-z_{k+1}} \\
					&\quad+ \gamma \sum_{k=0}^{n}(\xi_k+(-1)^{k}(1+z_k))^2\log\frac{1+z_{k+1}}{1+z_{k}}.
		\end{split}
		\label{e:1Denergymzero}
	\eeqn
\eprop

\begin{proof}
Let $u$ be given as in \eqref{e:axisymmfunc}. For such a finite partition $-1=z_0<z_1<\ldots<z_n<z_{n+1}=1$, the perimeter term in the energy is 
		\[
			\frac{1}{2}\int_\domainS |\nabla u| = 2\pi\sum_{k=1}^{n}\sqrt{1-z_k^2},
		\]
and the mass constraint 
		\[
			\frac1{4\pi}\int_\domainS u\,\exd\Htwo = \frac12\sum_{k=1}^{n+1}(-1)^k(z_k-z_{k-1}) = m.
		\]

Let $\xi$ be given by \eqref{e:xi}, a triangular wave function with slopes alternating between $-1-m$ and $+1-m$, and 
recall that $\xi(z)=(1-z^2)\partial_zv(z)$. Since $v$ is smooth near the poles,  elliptic regularity implies that 
 the values $\xi(-1)$ and $\xi(1)$ must vanish. Hence we have  $\sqrt{1-z^2}\partial_zv=\partial_\phi v\to0$ as $z\to\pm1$.

We first compute
	\beqn
			\xi_k:=\xi(z_k)=\sum_{i=1}^k(-1)^i(z_i-z_{i-1}) - m \sum_{i=1}^k(z_i-z_{i-1}),
		\label{xik}
	\eeqn
and 
	\[
		\xi(z)=\xi_k+a_{k+1}(z-z_k)
	\]
with $a_{k+1}=(-1)^{k+1}-m$ for $z\in(z_k,z_{k+1})$.
To compute the energy explicitly, we compute 
	\beqn
		\begin{split}
			\int_{z_k}^{z_{k+1}} \frac{(\xi_k+a_{k+1}(z-z_k))^2}{1-z^2}\,\exd z 
					&= -a_{k+1}^2\Delta_{k+1} 
+ \frac{(\xi_k-a_{k+1}z_k+a_{k+1})^2}2\log\frac{1-z_k}{1-z_{k+1}} \\
					&\quad + \frac{(\xi_k-a_{k+1}z_k-a_{k+1})^2}2\log\frac{1+z_{k+1}}{1+z_{k}},
		\end{split}
		\nonumber
	\eeqn
to find 
	\beqn
		\begin{split}
				\int_{-1}^1 \frac{\xi(z)^2}{1-z^2}\,\exd z 
						&= - 2 + 2m^2
+ \sum_{k=0}^{n}\frac{(\xi_k-((-1)^{k}+m)(1-z_k))^2}2\log\frac{1-z_k}{1-z_{k+1}} \\
						&\quad+ \sum_{k=0}^{n} \frac{(\xi_k+((-1)^{k}+m)(1+z_k))^2}2\log\frac{1+z_{k+1}}{1+z_{k}}.
		\end{split}
		\label{e:nonlocalintegral}
	\eeqn
Note that if any of the log-terms becomes infinite because of $z_0=-1$ or $z_{n+1}=1$,
the corresponding factor also vanishes because of the property $\xi(\pm1)=0$.

Combining the two terms, the energy $E_\gamma$ is given by
	\[
		E_\gamma(u) = 2\pi\sum_{k=1}^{n}\sqrt{1-z_k^2} + \gamma\, 2\pi \int_{-1}^1 \frac{\xi(z)^2}{1-z^2}\,\exd z,
	\]
	and  one obtains the explicit formulas \eqref{e:1Denergy} and \eqref{e:1Denergymzero}.
\end{proof}

\begin{remark}
In the case of two interfaces, that is when $n=2$ and $m=0$ the energy given by
		\beqn
			\begin{split}
				\frac1\pi E_\gamma(u)
						&=  - 4\gamma + 2\sum_{k=1}^{2}\sqrt{1-z_k^2}+ 4\gamma\log\frac{2}{1-z_{1}} + 4\gamma z_1^2\log\frac{1-z_1}{1-z_{2}}\\
						&\quad +4\gamma(1+z_1)^2\log\frac{1+z_2}{1+z_{1}} + 4\gamma\log\frac{2}{1+z_{2}}.
			\end{split}
			\nonumber
		\eeqn
captures the energy calculations in \cite{Top}. Moreover, since $z_2$ is determined by the mass constraint $z_2-z_1=1$, we can plot the energy $E_\gamma$ as a function of $z_1$ and $\gamma$ (See Figure \ref{energyplot}). Here one sees clearly that when $\gamma$ is small the energy $E_\gamma$ is minimized when either $z_1=0$ or $z_1=1$, i.e., the axisymmetric minimizer is the single cap with one interface; however, as $\gamma$ increases the double cap with interfaces located at $-1/2$ and $1/2$ is the configuration minimizing the energy.

	\begin{figure}[ht!]
  \centering
    \includegraphics[width=0.35\textwidth]{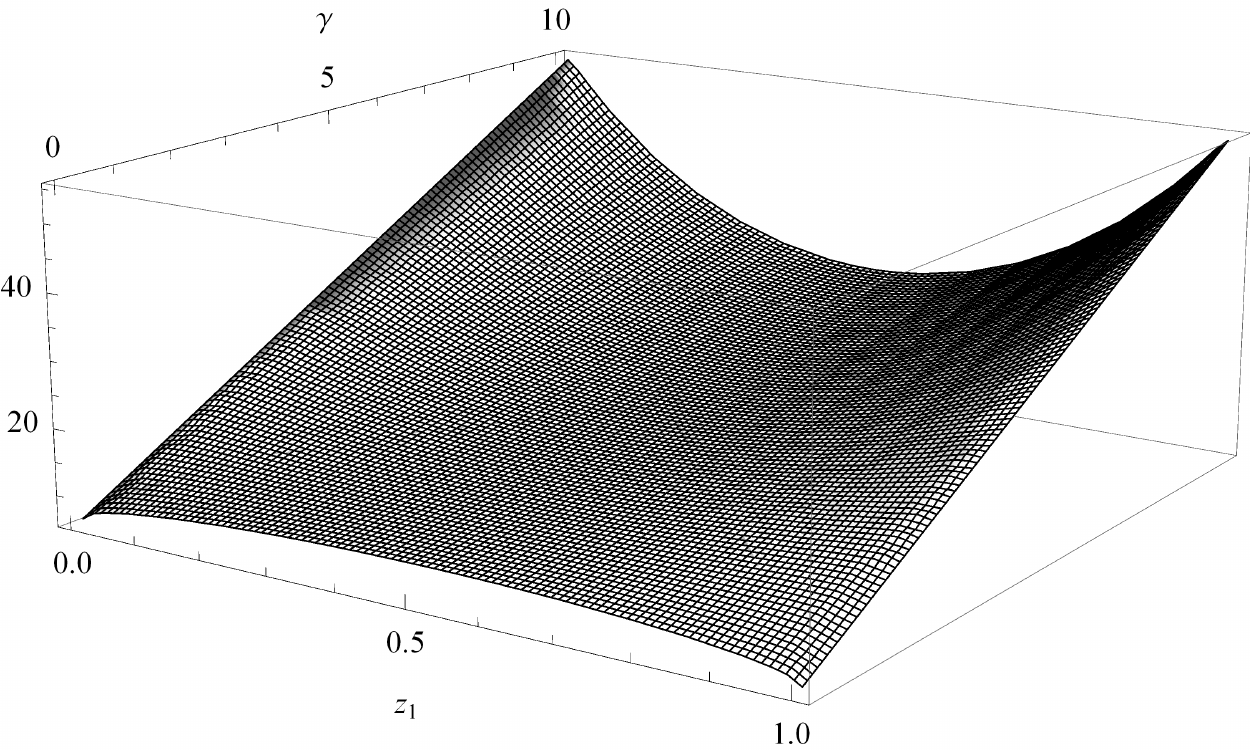}
  \caption{$E_\gamma$ as a function of the first interface $z_1$ and $\gamma$}
  \label{energyplot}
 	\end{figure}

\end{remark}

\section{Existence of a local minimizer restricted to the class of axisymmetric competitors}\label{sec:existence}

In this section, we  establish the existence of an axisymmetric local minimizer of $E_\gamma$ given in \eqref{e:1Denergy} with respect to axisymmetric perturbations for a fixed number of interfaces. We then note that this axisymmetric local minimizer is not only a critical point in its restricted class of axisymmetric perturbations but it is also a critical point of $E_\gamma$ over $BV(\domainS;\{\pm1\})$, namely, it is a solution of the Euler--Lagrange equation \eqref{e:firstvar1D}.
For simplicity, let us assume $m=0$, and without loss of generality let us consider an axisymmetric pattern with $2n$ interfaces which are symmetric with respect to the equator $z=0$.
Note that $E_\gamma$, then, must be minimized over 
$z=(z_1,z_2,\ldots,z_n)\in\R^n$ satisfying $0<z_1<\ldots<z_n<1$ and the mass constraint.
The mass constraint \eqref{e:massconstraint} defines an affine subspace of $\R^n$
	\[
		M:=\left\{z\in\R^n\colon \sum_{k=1}^n (-1)^k z_k=-1/2 \right\}
	\]
that has a nontrivial intersection with the simplex $Z=\{z\in\R^n \colon 0<z_1<\ldots<z_n<1\}$;
hence, the existence of a minimizer of $2n$ interfaces would follow if we could show that given any boundary point $z\in M\cap \partial Z$, 
one can reduce the energy by going into the interior $M\cap Z$ (e.g. Figure \ref{simplx} shows $M\cap Z$ when $n=3$).

	\begin{figure}[ht!]
  \centering
    \includegraphics[width=0.3\textwidth]{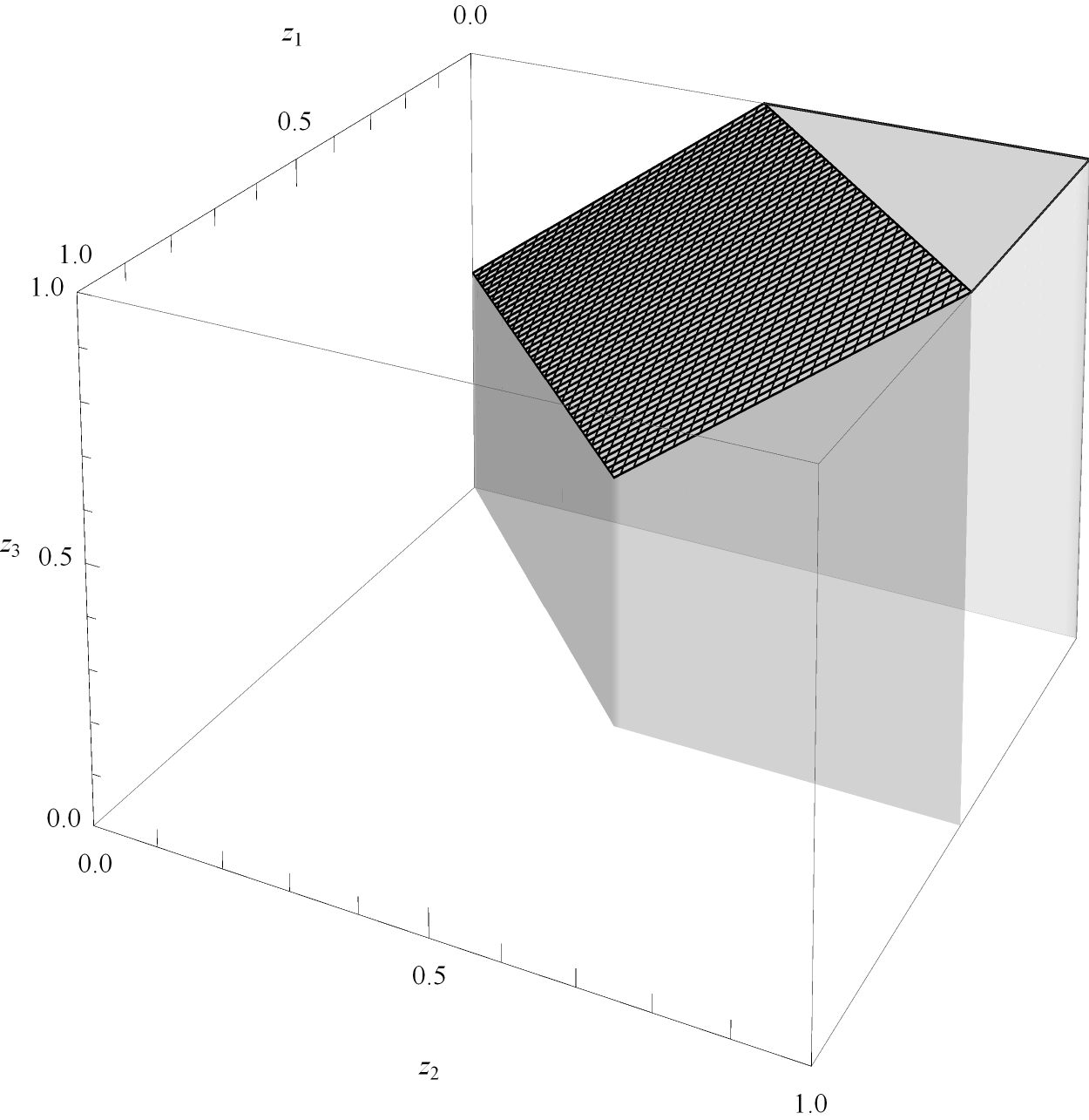}
  \caption{$M\cap Z$ for $n=3$ is shown by the shaded plane.}
  \label{simplx}
 	\end{figure}

Let $0=z_0\leq z_1\leq\ldots\leq z_n\leq z_{n+1}=1$ be a configuration
satisfying $z_{k}< z_{k+1} < z_{k+2} < z_{k+3}$ for some $k\in\{0,\ldots,n-2\}$.
Then we consider movements of the form $(z_{k+1},z_{k+2})\mapsto (z_{k+1}+t,z_{k+2}+t)$ for small $t$.
We call these movements {\em elementary moves}.
Note that the mass is automatically conserved under elementary moves.
For convenience, we introduce the variables
$\alpha=\frac{z_{k}+z_{k+1}}2$,
$\beta=\frac{z_{k+2}+z_{k+3}}2$,
and 
$x=\frac{z_{k+1}+z_{k+2}}2$, which are the roots of $\xi$ in the interval $[z_{k},z_{k+3}]$.
Under the elementary move we are considering, $\alpha$ and $\beta$ will be fixed, but $x$ will vary.
In fact, we will make $x$ our primary variable, and consider the other variables as a function of $x$.
For example, we have $z_{k+1}=\frac{\alpha+x}2$ and $z_{k+2}=\frac{\beta+x}2$.
Now we want to look at the dependence of the energy on $x$.
The part of the perimeter term that varies is
	\beqn
		e_\text{p}(x) = \sqrt{1-\big(\frac{\alpha+x}2\big)^2} + \sqrt{1-\big(\frac{\beta+x}2\big)^2},
		\nonumber
	\eeqn
where we omit the factor $2\pi$.
For the nonlocal term, it suffices to consider
	\beqn
		e_\text{nl}(x) = \int_{\alpha}^\beta \frac{\xi(z)^2}{1-z^2} \, \exd z.
		\nonumber
	\eeqn
Up to the factor $2\pi\gamma$, this is the contribution of the interval $(\alpha,\beta)$ to the nonlocal energy.

	\begin{lemma}\label{l:energy-piece}
We have
	\beqn
		e_\textnormal{nl}(x) = \frac{\alpha-x}{2} f\left(\frac{\alpha+x}2\right) + \frac{x-\beta}{2} f\left(\frac{x+\beta}2\right) + C,
		\nonumber
	\eeqn
where $C$ is a constant depending on $\alpha$ and $\beta$, and
	\beqn
		f(x) = (1-x)\log(1-x) + (1+x)\log(1+x).
		\nonumber
	\eeqn

	\end{lemma}
	
	\begin{proof}
Up to a sign, suppose $\xi(z)=z-\alpha$ in $(\alpha,z_{k+1})$, $\xi(z)=-(z-x)$ in $(z_{k+1},z_{k+2})$, and $\xi(z)=z-\beta$ in $(z_{k+2},\beta)$,
with $z_{k+1}=\frac{\alpha+x}2$ and $z_{k+2}=\frac{\beta+x}2$.
Then making use of the formula
	\beqn
		\int_a^b \frac{(z-c)^2}{1-z^2}\, \exd z 
		= -(b-a) + \frac{(1-c)^2}2 \log\frac{1-a}{1-b} + \frac{(1+c)^2}2 \log\frac{1+b}{1+a},
		\nonumber
	\eeqn
we can compute
	\beqn
	\begin{split}
		e_\text{nl}(x)
		&= -(\beta-\alpha) + \frac{(1-\alpha)^2}2 \log\frac{1-\alpha}{1-z_{k+1}} + \frac{(1+\alpha)^2}2 \log\frac{1+z_{k+1}}{1+\alpha} \\
		&\qquad + \frac{(1-x)^2}2 \log\frac{1-z_{k+1}}{1-z_{k+2}} + \frac{(1+x)^2}2 \log\frac{1+z_{k+2}}{1+z_{k+1}} \\
		&\qquad\qquad + \frac{(1-\beta)^2}2 \log\frac{1-z_{k+2}}{1-\beta} + \frac{(1+\beta)^2}2 \log\frac{1+\beta}{1+z_{k+2}}.
	\end{split}
	\nonumber
	\eeqn
Now letting
	\begin{equation}
			C:=-(\beta-\alpha)+\frac{(1+\alpha)^2}{2}\log(1-\alpha)-\frac{(1+\alpha)^2}{2}\log(1+\alpha)\\
				-\frac{(1-\beta)^2}{2}\log(1-\beta)+\frac{(1+\beta)^2}{2}\log(1+\beta)
		\nonumber
	\end{equation} 
and by writing the logarithms of the ratios as differences of logarithms and grouping the same logarithm terms together we obtain
	\beqn
	\begin{split}
		2(e_\text{nl}(x)-C)
		&= -(1-\alpha)^2 \log(1-z_{k+1}) + (1+\alpha)^2 \log(1+z_{k+1}) \\
		&\qquad\quad + (1-x)^2 \log\frac{1-z_{k+1}}{1-z_{k+2}} + (1+x)^2 \log\frac{1+z_{k+2}}{1+z_{k+1}} \\
		&\qquad\qquad\quad + (1-\beta)^2 \log(1-z_{k+2}) - (1+\beta)^2 \log(1+z_{k+2})\\
		&= [(1-x)^2-(1-\alpha)^2] \log(1-z_{k+1}) + [(1+\alpha)^2-(1+x)^2] \log(1+z_{k+1}) \\
		&\quad + [(1-\beta)^2-(1-x)^2] \log(1-z_{k+2}) + [(1+x)^2-(1+\beta)^2] \log(1+z_{k+2}),
	\end{split}
	\nonumber
	\eeqn
which is the desired conclusion.	
	\end{proof}

Our local analysis will depend on the contribution of the interval $(\alpha,\beta)$ to the total energy $E_\gamma$. To this end, let us denote it by
	\beqn
		e(x;\alpha,\beta,\gamma) := e_\text{p}(x) + \gamma\,e_\text{nl}(x)
		\label{e:littlee}
	\eeqn
where we also emphasize the dependence of $e$ on $\alpha$, $\beta$ and clearly on $\gamma$.

\begin{figure}[ht!]
     \begin{center}

        \subfigure[{\tiny$\alpha=0.1$, $\beta=0.2$, $\gamma=130$}]{
            \label{fig:ap1bp2q130}
            \includegraphics[width=0.2\linewidth]{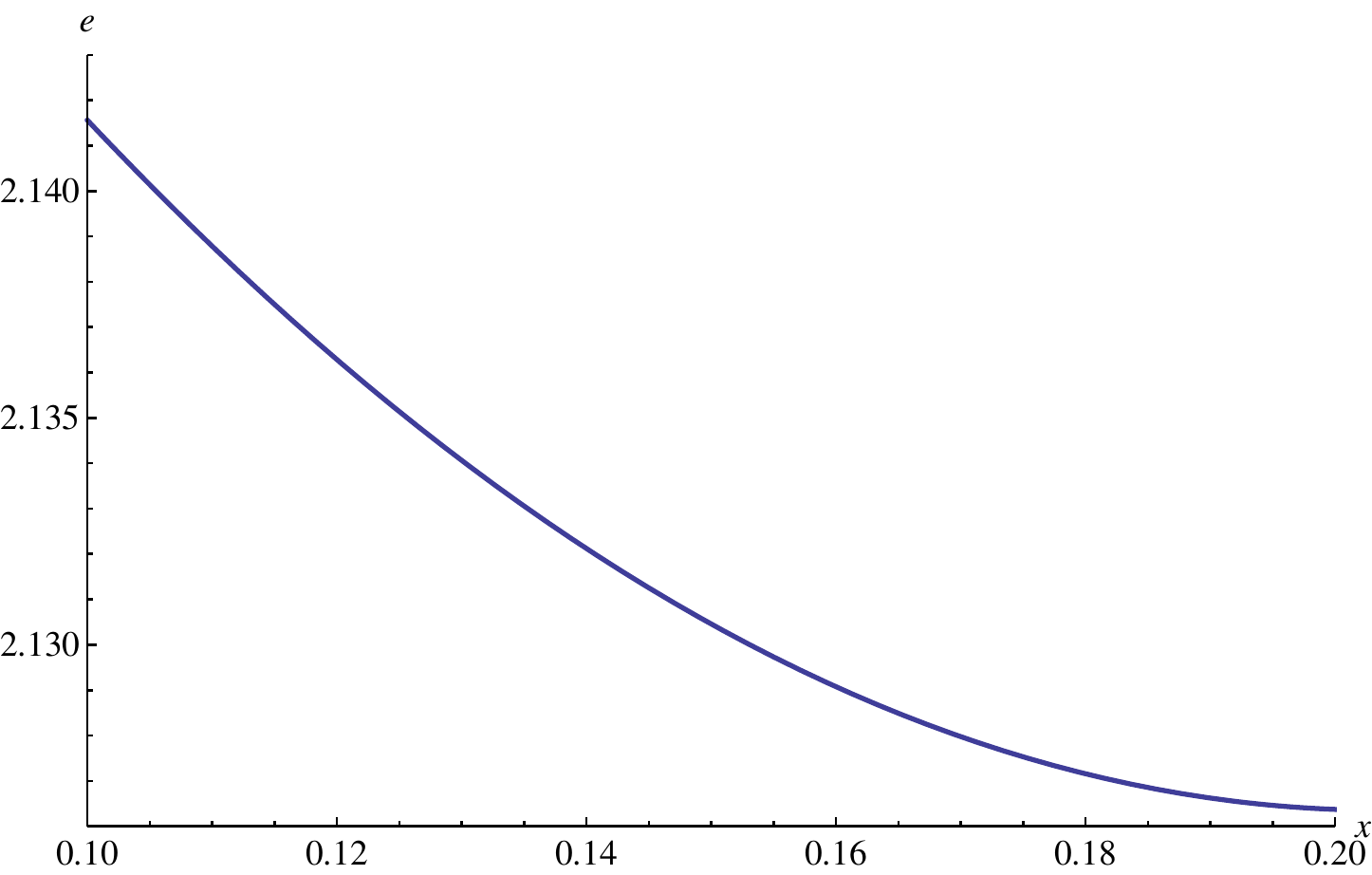}
        }
        \subfigure[{\tiny$\alpha=0.1$, $\beta=0.2$, $\gamma=200$}]{
           \label{fig:ap1bp2q250}
           \includegraphics[width=0.2\linewidth]{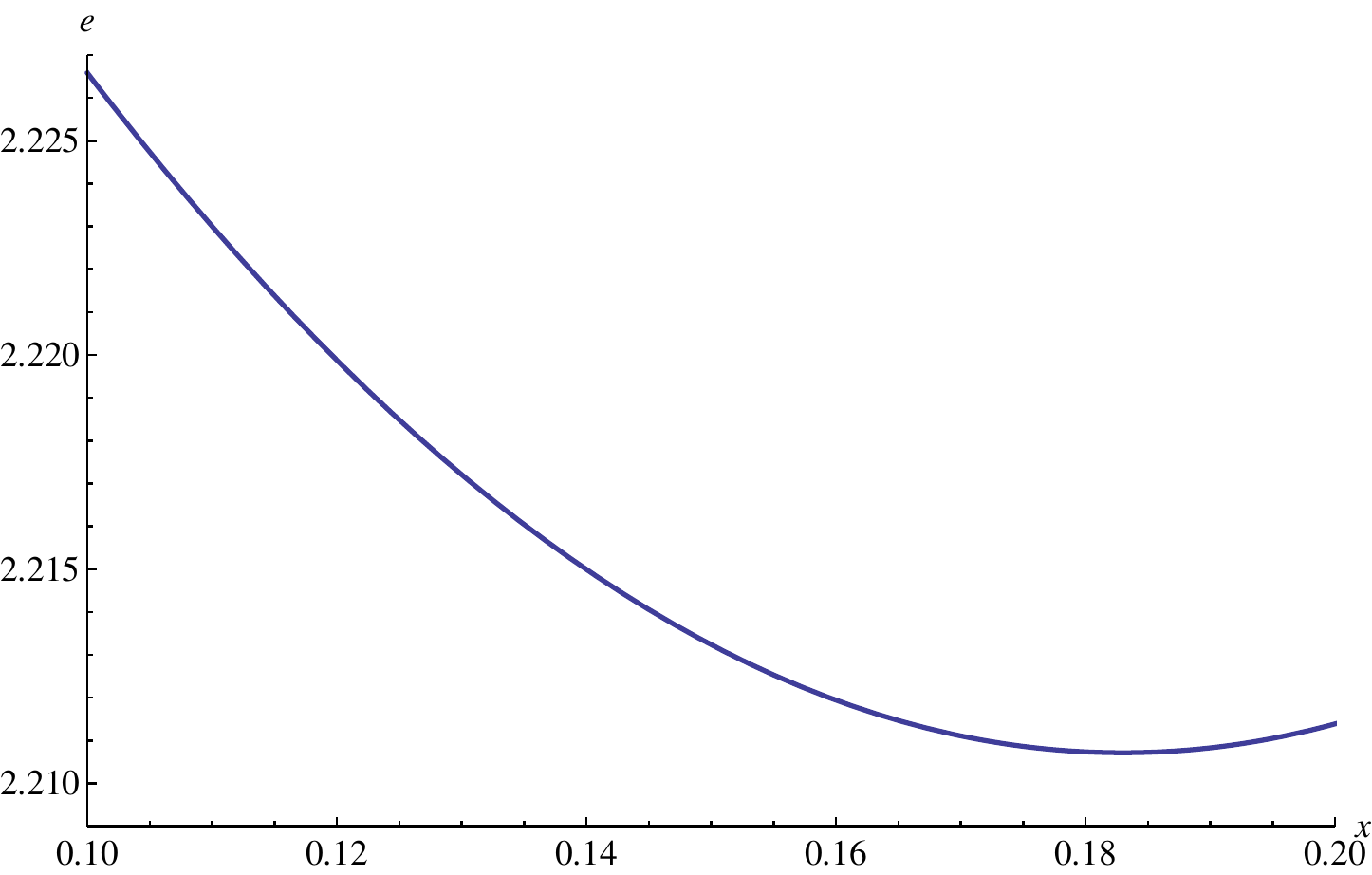}
        }
        \subfigure[{\tiny$\alpha=0.1$, $\beta=0.2$, $\gamma=350$}]{
            \label{fig:ap1bp2q350}
            \includegraphics[width=0.2\linewidth]{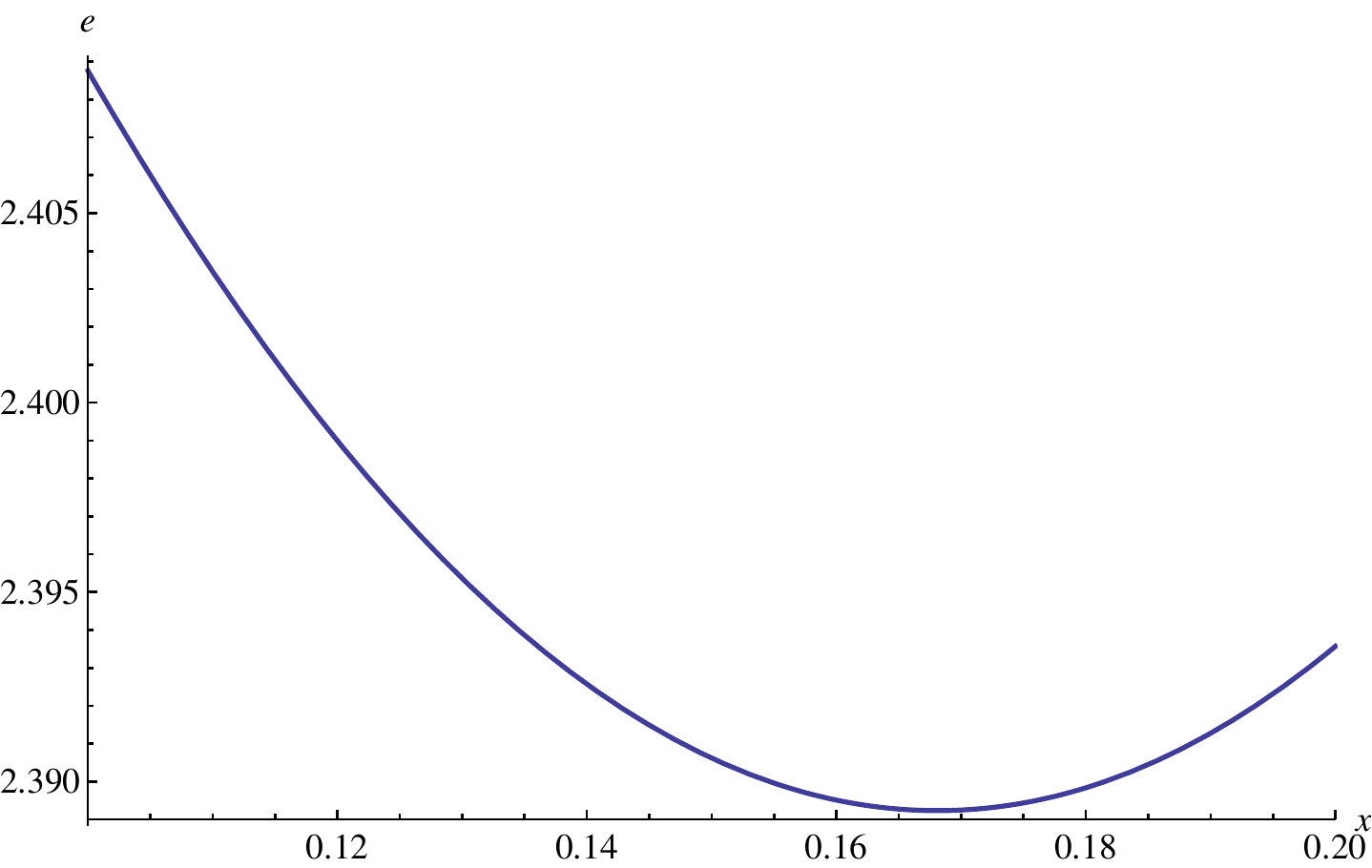}
        }\\ 
        \subfigure[{\tiny$\alpha=0.6$, $\beta=0.7$, $\gamma=350$}]{
            \label{fig:ap6bp7q350}
            \includegraphics[width=0.2\linewidth]{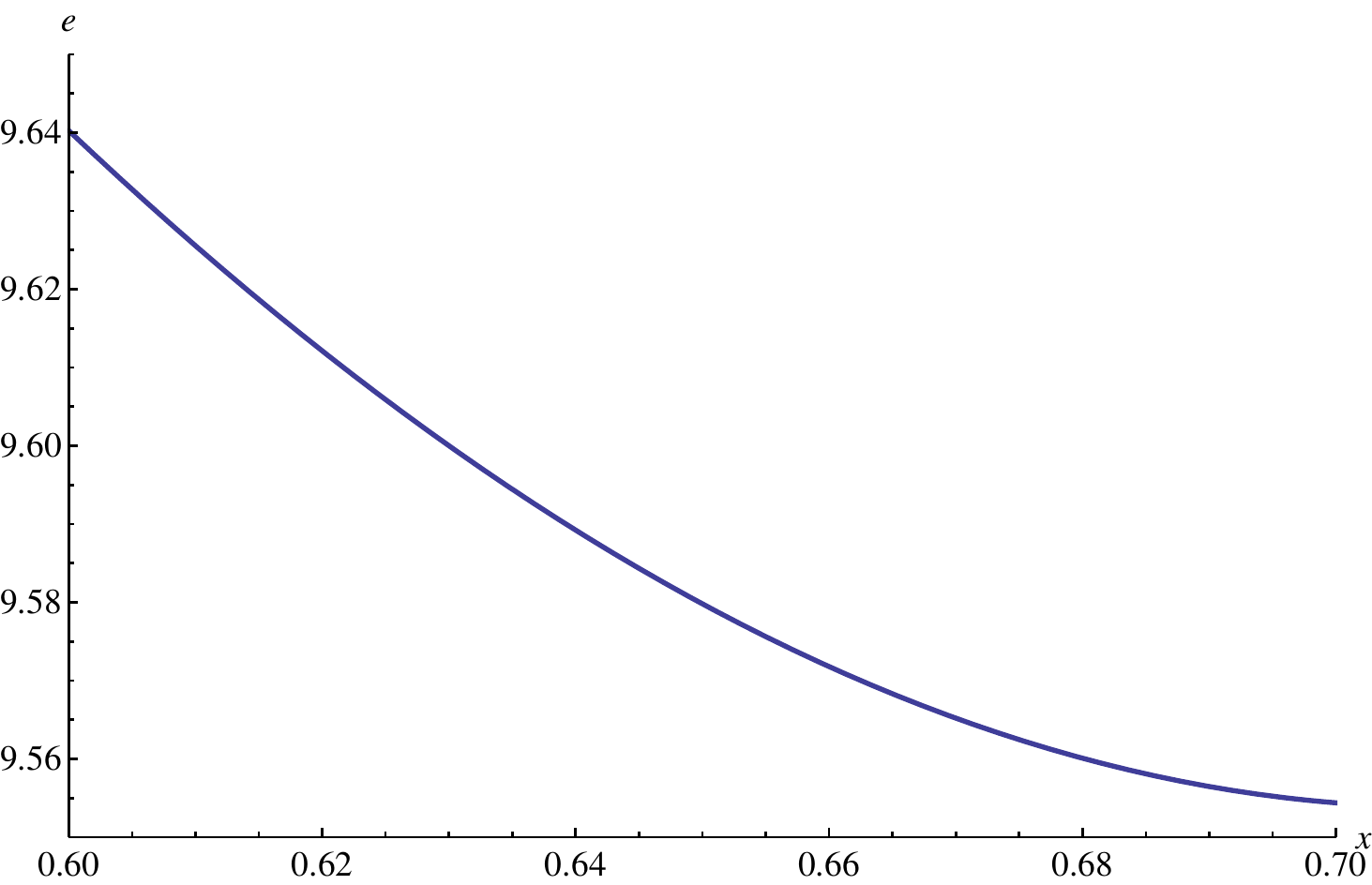}
        }
        \subfigure[{\tiny$\alpha=0.6$, $\beta=0.7$, $\gamma=750$}]{
           \label{fig:ap6bp7q750}
           \includegraphics[width=0.2\linewidth]{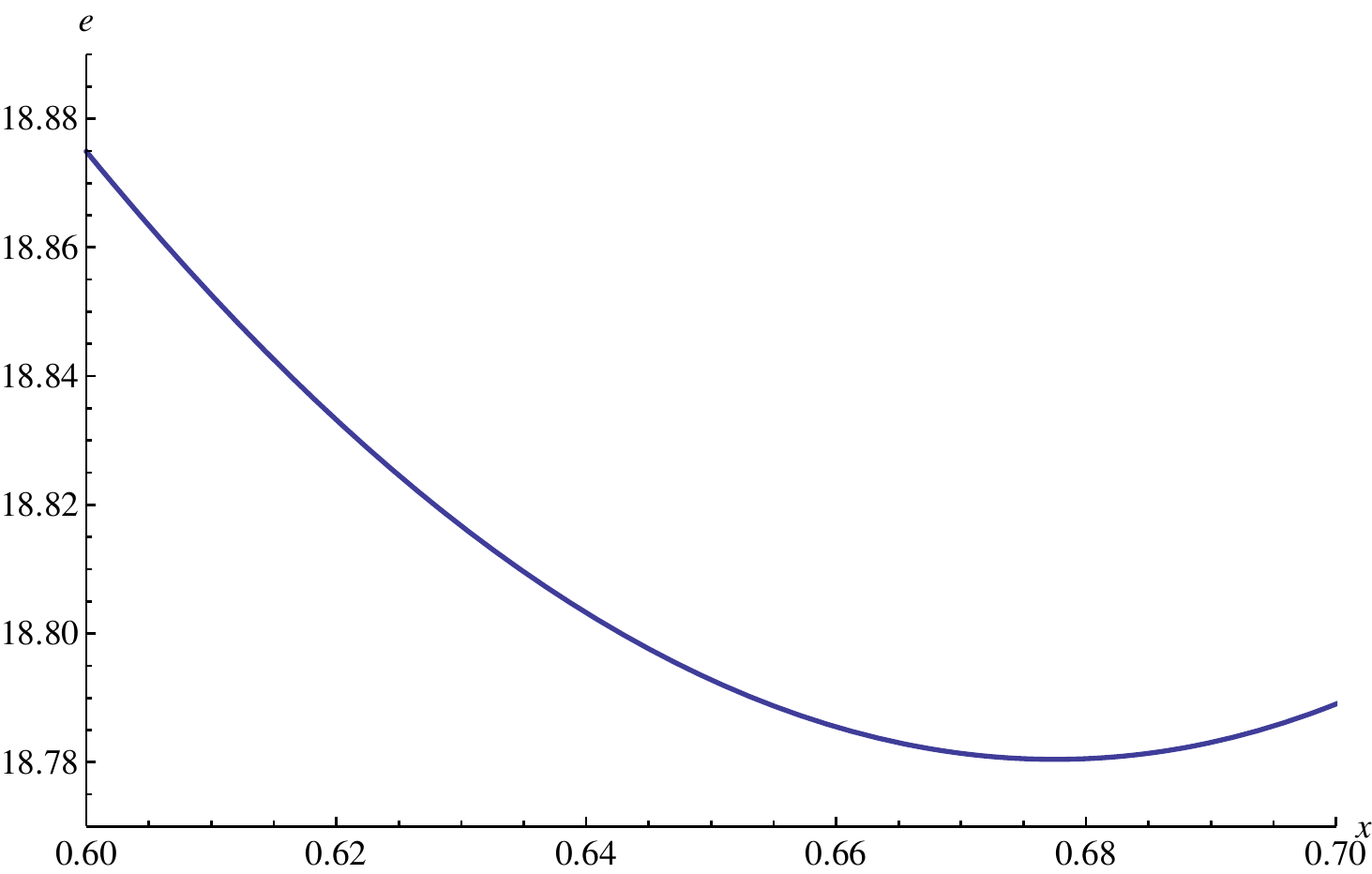}
        }
        \subfigure[{\tiny$\alpha=0.6$, $\beta=0.7$, $\gamma=10000$}]{
            \label{fig:ap6bp7q10000}
            \includegraphics[width=0.2\linewidth]{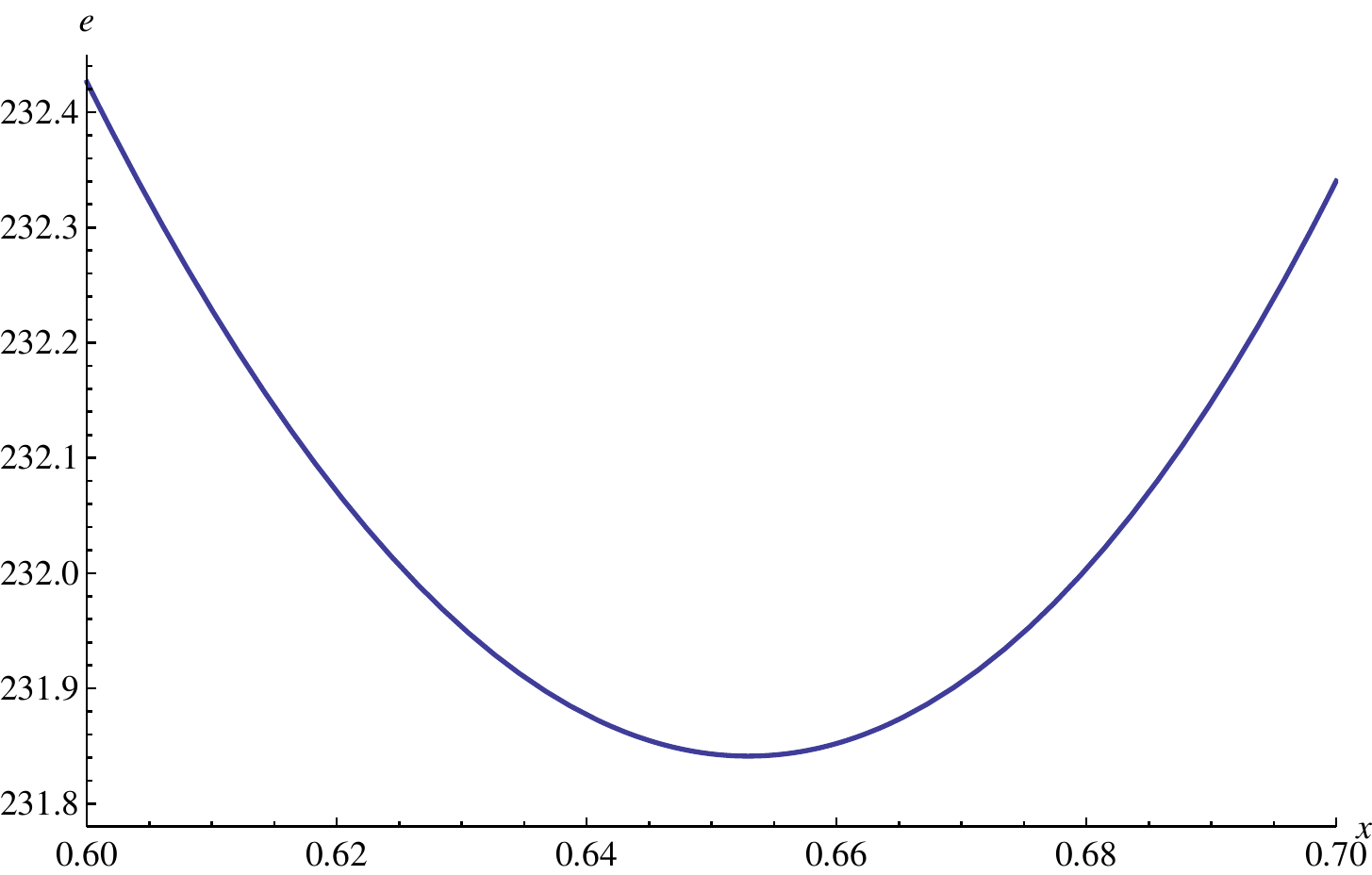}
        }\\ 
        \subfigure[{\tiny$\alpha=0.9$, $\beta=1$, $\gamma=500$}]{
            \label{fig:ap9b1q500}
            \includegraphics[width=0.2\linewidth]{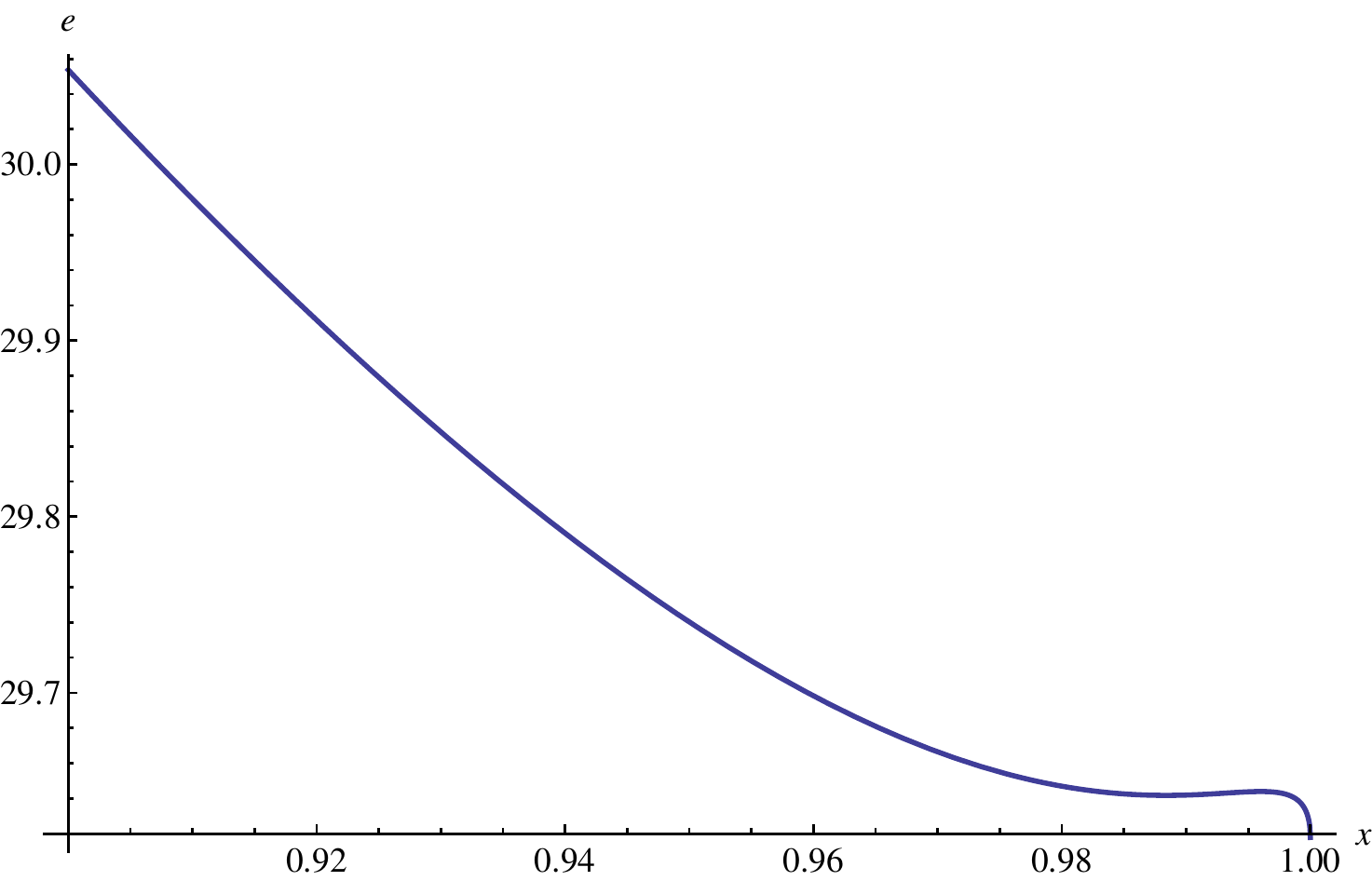}
        }
        \subfigure[{\tiny$\alpha=0.9$, $\beta=1$, $\gamma=1000$}]{
            \label{fig:ap9b1q1000}
            \includegraphics[width=0.2\linewidth]{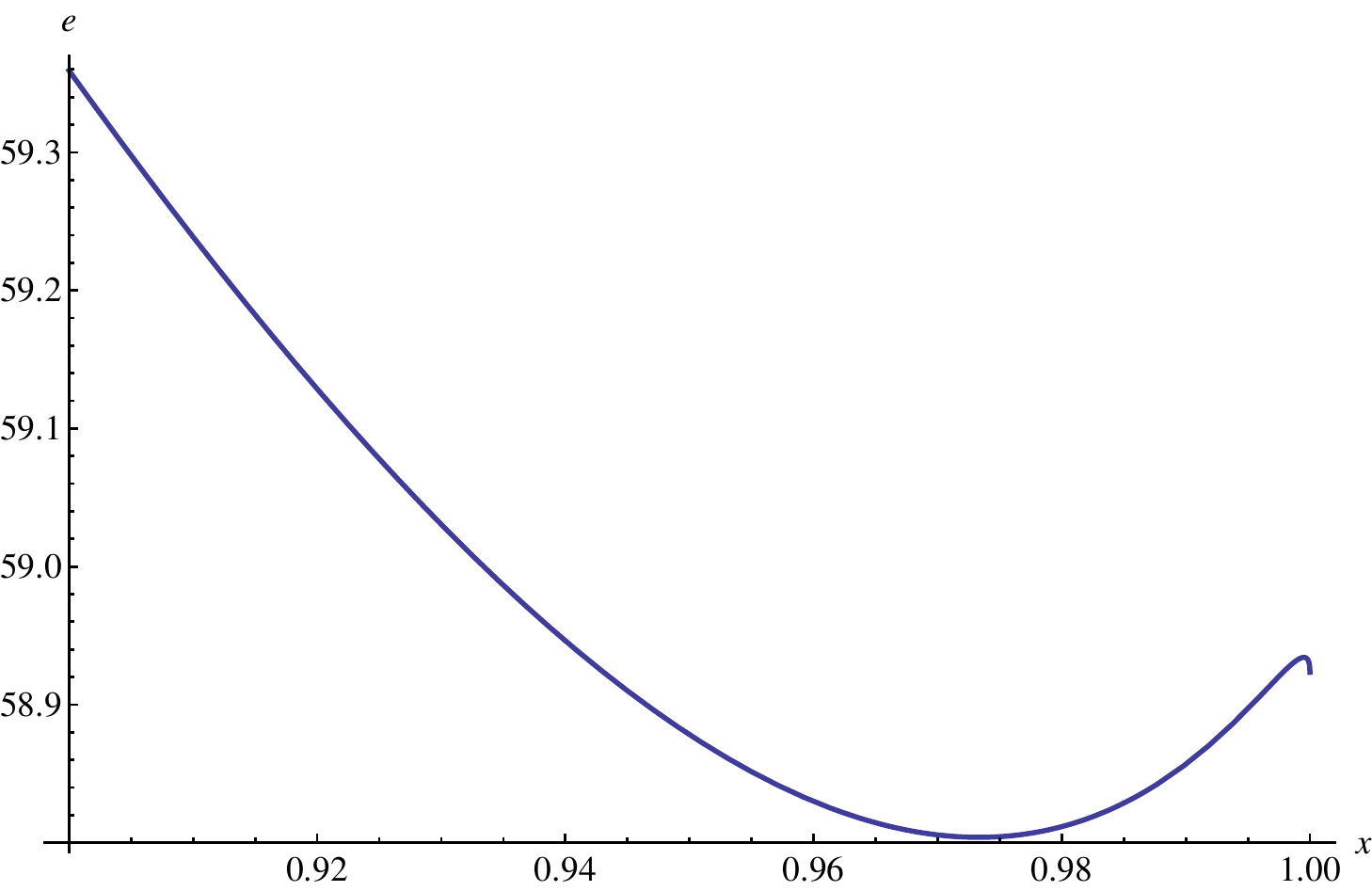}
        }
        \subfigure[{\tiny$\alpha=0.9985$, $\beta=1$, $\gamma=400000$}]{
            \label{fig:ap9985b1q400000}
            \includegraphics[width=0.2\linewidth]{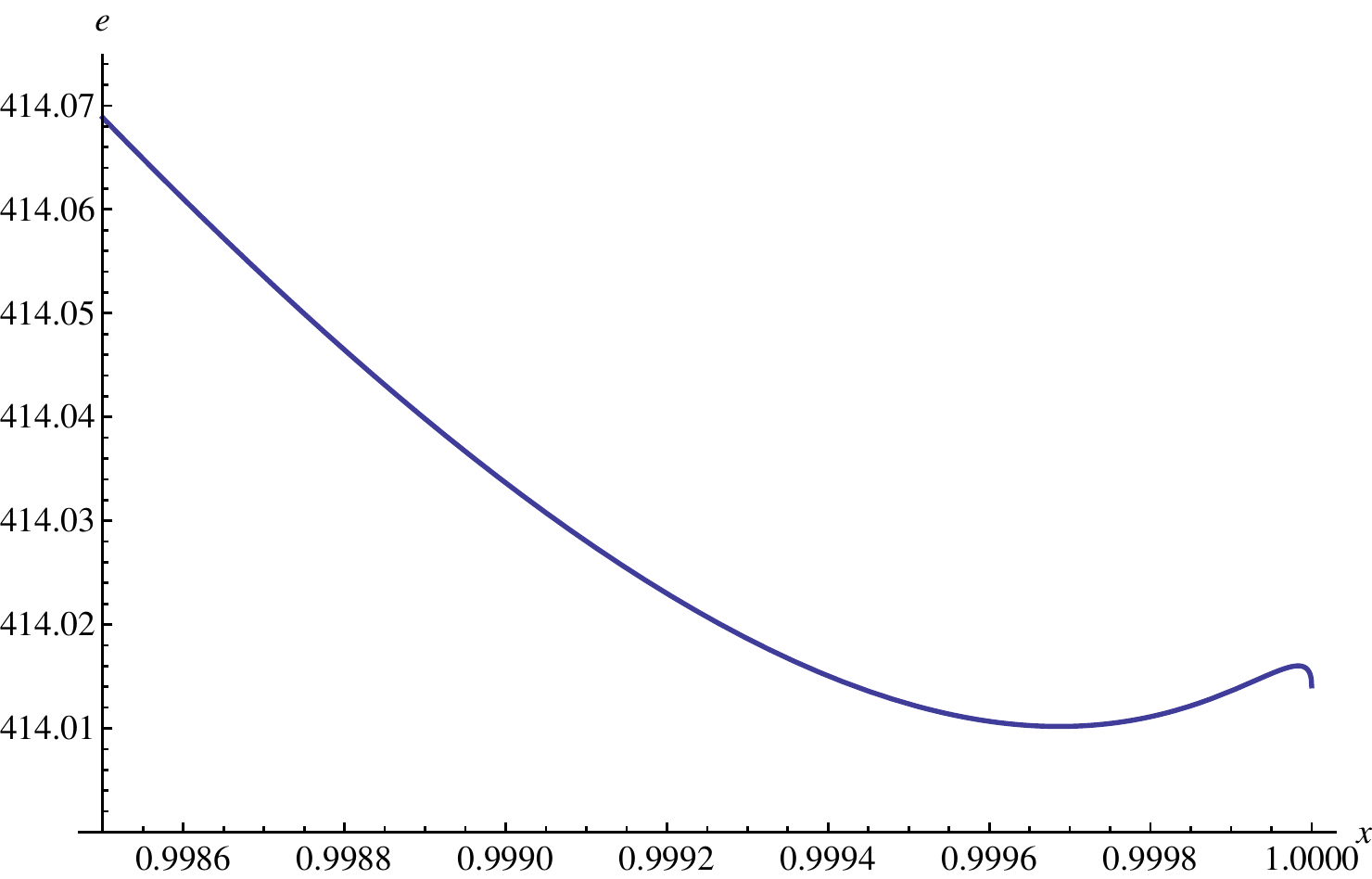}
        }

    \end{center}
    \caption{Examples of $e(x;\alpha,\beta,\gamma)$ with varying $\alpha$, $\beta$ and $\gamma$. The first and second rows show how the energy contribution changes under an elementary move in a typical situation. The third row illustrates the behavior near the poles.}
   \label{fig:variousegraphs}
\end{figure}

Note that when $\beta$ is $1$, an interesting situation arises (see Figures \ref{fig:ap9b1q500}--\ref{fig:ap9985b1q400000}).
Because the perimeter term has a derivative tending to $-\infty$ like $-x/\sqrt{1-x^2}$,
and this cannot be counterbalanced by the nonlocal term, any interface that is too close to a pole would be sucked into the pole and disappear.
However, if $\gamma$ is large, poles are repulsive at larger distance. The next technical lemma characterizes the behavior of $e(x;\alpha,\beta,\gamma)$ near the boundary $\beta=1$.

\begin{lemma}\label{lem:poles}
		For any $\alpha\in(0,1)$ there exists $x\in(\alpha,1)$ such that when $\gamma$ is large enough
			\[
				\min_{x\in(\alpha,1)} e(x;\alpha,1,\gamma)<L,
			\]
		where
			\[
				L:=\lim_{x\arrow 1} e(x;\alpha,1,\gamma)
			\]
		is a positive constant depending on $\alpha$ and $\gamma$.
\end{lemma}

\begin{proof}
	Let $\alpha\in(0,1)$ be fixed. To simplify the notation we will drop the dependence on $\alpha$. First, note that, because of the structure of the nonlocal contribution we have
		\[
			\lim_{x\arrow \alpha} e_\text{nl}(x) = \lim_{x\arrow 1} e_\text{nl}(x)
		\]
and
		\[
			e_\text{nl}(x) - \lim_{x\arrow 1}e_\text{nl}(x) < 0
		\]
for all $x\in(\alpha,1)$.
Also, writing out the perimeter term explicitly, we see that
		\[
			e_\text{p}(x) - \lim_{x\arrow 1}e_\text{p}(x) > 0
		\]
for all $x\in(\alpha,1)$.

Now consider
	\beqn
		e(x;\alpha,1,\gamma) - L = \left(e_\text{p}(x) - \lim_{x\arrow 1}e_\text{p}(x)\right) + \gamma\,\left(e_\text{nl}(x) - \lim_{x\arrow 1}e_\text{nl}(x)\right).
		\label{e:eminusL}
	\eeqn
Take $x=(1+\alpha)/2$. Since $e_\text{p}^\pr((1+\alpha)/2)<0$ and is finite, for large enough $\gamma$ the second term on the right-hand side of \eqref{e:eminusL} is dominant and yields that $e((1+\alpha)/2;\alpha,1,\gamma) - L < 0$. On the other hand, since the derivative of $e_\text{p}^\pr$ is $-\infty$ at $x=1$ and $e(1;\alpha,1,\gamma)-L=0$, for $x_2$ sufficiently close to $1$ the first term in \eqref{e:eminusL} dominates and $e(x_2;\alpha,1,\gamma) - L>0$. Moreover, as $\lim_{x\arrow \alpha} e_\text{nl}(x) = \lim_{x\arrow 1} e_\text{nl}(x)$ and $e_\text{p}(x)$ is a strictly decreasing function on $(\alpha,1)$, for some $x_1$ sufficiently close to $\alpha$ we have $e(x_1;\alpha,1,\gamma) - L>0$. This means that the function
	 \[
	 	e(x;\alpha,1,\gamma) - L
	 \]
is positive at $x=x_1>\alpha$, negative at $x=(1+\alpha)/2$ and again positive at $x=x_2<1$. Thus by differentiability of $e(x;\alpha,1,\gamma)$ on $(\alpha,1)$ the result follows.
\end{proof}
	
Now we prove the existence of an axisymmetric local minimizer. 

\bthm\label{thm:existence}
Given $n\in\mathbb{N}$ the energy $E_\gamma$ admits a local minimizer in the set $M\cap Z$ provided $\gamma>0$ is sufficiently large.
\ethm

\begin{proof}
We start with a point at the boundary of $M\cap Z$, and show that there exists an elementary move that brings us to the interior and reduces the energy. First note that since $\dim(Z)=n$ and $\dim (M)=n-1$ we have that $\dim(M\cap\pt Z)=n-2$. Let $z\in M\cap\pt Z$ and suppose that $z$ is given by $0<z_1<z_2<\ldots<z_{n-2}<z_{n-1}<z_{n}=1$. Let
	\[
		\alpha=\frac{z_{n-2}+z_{n-1}}{2},\quad x=\frac{z_{n-1}+z_n}{2},\quad \beta=\frac{z_n+1}{2}=1.
	\]
Consider the elementary move $\sigma_t:(z_{n},z_{n-1})\mapsto (z_{n}-t,z_{n-1}-t)$ for $t>0$ small. Note that this move creates an interface at $z_n$; hence, moves the point $z$ into the interior of $M\cap Z$. Since for small $t>0$ this elementary move creates an interface near the pole $z=1$, Lemma \ref{lem:poles} shows that when $\gamma$ is sufficiently large there exists $t_0>0$ such that the image of $x$ under $\sigma_{t_0}$, namely $x-t_0$, corresponds to $\argmin_{x\in(\alpha,1)} e(x;\alpha,1,\gamma)$. Thus the energy $E_\gamma$ has an axisymmetric local minimizer in $M\cap Z$ with interfaces given by $0<z_1<\ldots<z_{n-2}<z_{n-1}-t_0<z_n-t_0<1$. 

Now let $z\in M\cap\pt Z$ determined by interfaces located at $0\leq z_1 \leq z_2 \leq \ldots \leq z_{n}<1$ where $z_k=z_{k+1}$ for a single index $k\in\{0,1,\ldots,n-2\}$ and for any other index the inequality is strict, i.e., $z_i<z_{i+1}$ for $i\neq k$. Let
	\[
		\alpha=\frac{z_{k-2}+z_{k-1}}{2},\quad x=\frac{z_{k-1}+z_k}{2},\quad \beta=\frac{z_{k+1}+z_{k+2}}{2}.
	\]
Since $k\leq n-2$, $\beta$ is away from the pole $z=1$.	Consider the elementary move $\sigma_t:(z_{k-1},z_{k})\mapsto (z_{k-1}-t,z_{k}-t)$. Note that $\sigma_t$ keeps $z_{k+1}$ fixed and by moving the line segment between $z_{k-1}$ and $z_k$ it creates another interface at $z_k-t$ for $t>0$ different than $z_{k+1}$. Thus, $\sigma_t$ moves the point $z$ on the boundary of $M\cap Z$ towards its interior.

A basic observation on why this elementary move reduces the energy if we place the interfaces more or less evenly is that 
	\beqn
	\begin{split}
		2e_\text{nl}^\pr(x) &= - f\left(\frac{\alpha+x}2\right) + f\left(\frac{x+\beta}2\right) + \frac{\alpha-x}2 f^\pr\left(\frac{\alpha+x}2\right) + \frac{x-\beta}2 f^\pr\left(\frac{x+\beta}2\right) \\
		       &= f_2(x) - f_1(x),
	\end{split}
	\nonumber
	\eeqn
where $f_1(x)$ is the linear Taylor approximation of $f$ evaluated at $x$, with the base point of the Taylor expansion located at $\frac{\alpha+x}2$,
and $f_2(x)$ is an analogous approximation with the base point at $\frac{x+\beta}2$.
Since $f$ is a strictly convex function, it is clear that when $x$ is near the left endpoint $\alpha$ then $e_\text{nl}^\pr(x)<0$, and similarly if $x$ is close to the right endpoint $\beta$
then $e_\text{nl}^\pr(x)>0$. What this means is that the nonlocal energy increases under an elementary move if one tries to place a strip of material too far from the middle of the allowed space. On the other hand, the perimeter tries to move the strip away from the equator since $e_\text{p}$ is a decreasing function on $(\alpha,\beta)$. However, since $e_\text{p}^\pr$ is bounded on $(\alpha,\beta)$, we can choose $\gamma$ large enough such that
	\[
		\argmin_{x\in(\alpha,\beta)} e(x;\alpha,\beta,\gamma) < \beta.
	\]
Examples of choices of $\gamma$, depending on $\alpha$ and $\beta$, satisfying this condition are illustrated in Figures \ref{fig:ap1bp2q130}--\ref{fig:ap6bp7q10000}. So, taking $t_0$ such that $x-t_0=\argmin_{x\in(\alpha,\beta)} e(x;\alpha,\beta,\gamma)$ we see that the elementary move $\sigma_{t_0}$ reduces the energy $E_\gamma$ to a local minimum.
\end{proof}

\bigskip

Next we will show that the critical points found in this restricted class consisting of only axisymmetric patterns are actually contained in the set of critical points of $E_\gamma$ with respect to \emph{all} perturbations.

\bprop\label{p:critequiv}
	Let $u(z)$ be an axisymmetric critical point of $E_\gamma$ with respect to one dimensional perturbations in the $z$-variable. Then $u(z)$ as a function on $\domainS$ is a critical point of $E_\gamma$ in the sense that it is a solution of the Euler--Lagrange equation \eqref{e:firstvar1D}.
\eprop

\begin{proof}
 Suppose $u(z)$ has $n$-interfaces located at $-1<z_1<\cdots<z_n<1$. Let $A$ be the axisymmetric set where $u=1$, and let $\Gamma_k$ be the components of the boundary of $A$, i.e., $\pt A = \bigcup_{k=1}^n \Gamma_k$ with $\Gamma_k$ located at $z_k$ in $z$-variables. Since the geodesic curvature of each component $\Gamma_k$ is equal to a constant, we have 
 	\beqn
 			\kappa_g(z_k)+4\gamma v(z_k) = \lambda_k
 		\label{e:constongammak}
 	\eeqn
for some constant $\lambda_k$.

Following \cite{CS2},  the first variation of $E_\gamma$ is calculated  by taking perturbations of the set $A$ with respect to its outer normal vector. Indeed, for any smooth function $f:\pt A\arrow\mathbb{R}$ such that $\int_{\pt A} f(x)\,\exd\Htwo=0$, we define perturbations of $A$ by letting $A$ flow via the gradient field defined as
	\[
		X(x)=f(x)\, \nu(x)
	\]
for $x\in\pt A$ where $\nu$ denotes the unit outer normal to $A$. Then, the calculations yield the following weak formulation:
	\begin{equation}\label{eq-weak}
		\int_{\pt A} (\kappa_g(x)+4\gamma v(x))f(x)\,\exd\Hone = 0.
	\end{equation}

To find the first variation of $E_\gamma$ about an axisymmetric pattern with respect to axisymmetric perturbations \emph{only}, $f$ is taken to be a constant on each boundary component $\Gamma_k$. Then for an axisymmetric  $u(z)$ with $n$-interfaces, criticality implies the following condition:   For all $c_1,\ldots,c_n\in\mathbb{R}$ such that 
	\beqn
		\sum_{k=1}^n c_k\,\Hone(\Gamma_k) = 0, 
		\label{e:fdiscrete}
	\eeqn
we have
	\beqn
		\sum_{k=1}^n c_k\lambda_k\,\Hone(\Gamma_k) = 0.
		\label{e:discretefirstvar}
	\eeqn
Note that the condition 	(\ref{e:discretefirstvar}) is simply a discrete version of the weak condition (\ref{eq-weak}). 
	
To show that $u(z)$ is also a critical point with respect to {\it general} perturbations, i.e. a solution to (\ref{firstvar}), 
 we need to show that there exists a constant $\lambda$ such that
	\[
		\kappa_g(z_k)+4\gamma v(z_k) = \lambda_k = \lambda \qquad   \hbox{\rm  for all $k=1,\ldots,n$}.
	\]
To this end, suppose that  for some $i \neq j$ we have $\lambda_i\neq\lambda_j$. Let $c>0$ be an arbitrary constant, and choose the constants $c_1,\ldots,c_n$ to be all zero except for 
	\[
		c_i=c \quad {\rm and} \quad c_j=-\frac{c\Hone(\Gamma_i)}{\Hone(\Gamma_j)}.
	\]
Clearly the set of numbers $c_1,\ldots, c_n$ satisfies \eqref{e:fdiscrete}, and hence by  \eqref{e:discretefirstvar} we have 
	\[
		(\lambda_i-\lambda_j)c\,\Hone(\Gamma_i)=0.
	\]
But this contradicts the fact that $\lambda_i-\lambda_j\neq 0$. Hence we must have $\lambda_k=\lambda$ 
for all $k=1,\ldots,n$.
\end{proof}

\section{Rigidity of the Criticality Condition}\label{sec:rigidity}

By Theorem \ref{thm:existence} and Proposition \ref{p:critequiv} we know there exists at least one critical point of $E_\gamma$ with arbitrary number of interfaces  
provided $\gamma$ is sufficiently large. In other words,  the equation \eqref{firstvar}, or in particular \eqref{e:firstvar1D} has a solution. However, due to the curved nature of the sphere, this criticality condition \eqref{firstvar} is a rather rigid condition. In this section we exploit this rigidity to obtain quantitative results on the axisymmetric critical points. 

We first show that for any fixed $\gamma>0$ one only has axisymmetric critical patterns with a finite bounded number of interfaces.
	\bprop[Number of interfaces]\label{p:numberofinterfaces}
		If for any given $\gamma>0$ an axisymmetric pattern $u$ with $n$ interfaces is a critical point of $E_\gamma$ then there exists a number $N_{\gamma}\in\mathbb{N}$ depending on $\gamma$ such that $n\leq N_{\gamma}$, that is, the number of interfaces that $u$ has is bounded from above.
	\eprop
	
	\begin{proof}
Let $\gamma>0$ be fixed and let $u(z)$ be an axisymmetric critical point with at least $n$ interfaces. Let $\delta_1>0$ be an arbitrary number. Assume, for a contradiction, that 
	\[
		|z_i-z_j|<\delta_1
	\]
for all $i,j=1,\ldots,n$. Then, in particular, two consecutive interfaces $z_i$ and $z_{i+1}$ are at most $\delta_1$ apart, i.e.,
	\[
		|z_{i+1}-z_i|<\delta_1.
	\]
Since $u(z)$ is a critical point, by \eqref{e:firstvar1D}, we have that
	\beqn
		\kappa_g(z_{i+1})-\kappa_g(z_{i})+4\gamma\,(v(z_{i+1})-v(z_{i}))=0.
		\label{critdiff-i}
	\eeqn

Let $\epsilon>0$ be another arbitrary number. Since $v(z)$ is uniformly continuous on $[-1,1]$, there exists $\delta_2>0$ such that
	\[
		|v(z_{i+1})-v(z_i)| < \frac{\epsilon}{8\gamma}
	\]
provided $|z_{i+1}-z_i|<\delta_2$.

Now we are going to look at two cases:

\noindent {\it Case 1.} (The interfaces at $z_i$ and $z_{i+1}$ are on the same hemisphere.) Since $z_i,z_{i+1}\in(-1,0)$ or $z_i,z_{i+1}\in(0,1)$, recalling \eqref{e:geod-curv}, the geodesic curvatures $\kappa_g(z_i)$ and $\kappa_g(z_{i+1})$ have the opposite sign. Therefore, for $\epsilon>0$, there exists $\delta_3>0$ such that
	\[
		|\kappa_g(z_{i+1})-\kappa_g(z_i)| > \epsilon
	\]
for $|z_{i+1}-z_i|<\delta_3$. Let
	\[
		\delta:=\min\{\delta_1,\delta_2,\delta_3\}.
	\]
Then, for $|z_{i+1}-z_i|<\delta$, we have
	\[
		|\kappa_g(z_{i+1})-\kappa_g(z_i)| \geq \epsilon\qquad\text{and}\qquad 4\gamma|v(z_{i+1})-v(z_i)|<\epsilon/2;
	\]
hence, 
	\[
		(\kappa_g(z_{i+1})-\kappa_g(z_{i}))+4\gamma\,(v(z_{i+1})-v(z_{i}))
	\]
is either strictly positive or strictly negative. This contradicts the equation \eqref{critdiff-i} and the criticality of $u$.

\medskip

\noindent {\it Case 2.} (The interfaces at $z_i$ and $z_{i+1}$ are on different hemispheres.) In this case, we assume, without loss of generality, that $-1<z_i<0<z_{i+1}<1$. This implies that $\kappa_g$ has the same sign at $z_i$ and $z_{i+1}$. Also, by \eqref{critdiff-i}, we have
	\beqn		\frac{\kappa_g(z_{i+1})-\kappa_g(z_i)}{z_{i+1}-z_{i}}+4\gamma\,\frac{v(z_{i+1})-v(z_{i})}{z_{i+1}-z_{i}}=0.
	 \label{critdifffrac}
	\eeqn
	
Looking at the first term, we see that for $\epsilon>0$, there exists $\delta_4>0$ such that
	\beqn
		\left| \frac{\kappa_g(z_{i+1})-\kappa_g(z_i)}{z_{i+1}-z_{i}}\right|=\left| \frac{-\frac{z_{2i}}{\sqrt{1-z_{2i}^2}}-\frac{z_{2i-1}}{\sqrt{1-z_{2i-1}^2}}}{z_{2i}-z_{2i-1}} \right| > 1-\epsilon
		\label{e:diff-i-frac2}
	\eeqn
when $|z_{i+1}-z_i|<\delta_4$. Also, since $\pt_z v(z)$ is uniformly continuous on $[-1,1]$, there exists $\delta_5>0$ such that
	\beqn
		\left|\frac{v(z_{i+1})-v(z_{i})}{z_{i+1}-z_{i}}\right|<\frac{\epsilon}{8\gamma}.
		\label{e:diff-i-frac1}
	\eeqn 
for $|z_{i+1}-z_i|<\delta_5$. Letting
	\[
		\delta:=\min\{\delta_1,\delta_4,\delta_5\}
	\]
this time, and combining \eqref{e:diff-i-frac2} and \eqref{e:diff-i-frac1} we reach a contradiction with the equation \eqref{critdifffrac} and the criticality of $u$.

These two cases show that for fixed $\gamma>0$, any two interfaces $z_i$ and $z_j$ of an axisymmetric critical point cannot be arbitrarily close to each other.  Hence we conclude that $n$ is bounded from above. Note, of course, that our argument eliminates the possibility of infinitely many interfaces.
\end{proof}

The next result emphasizes the rigidity of \eqref{e:firstvar1D} insofar as axisymmetric patterns maintaining their criticality for an interval of $\gamma$-values can only have one or two interfaces.

	\bprop\label{p:intervalofgamma}
	 The only axisymmetric critical points of $E_\gamma$ for an interval of $\gamma$-values are the single cap and the symmetric double cap, that is, the critical points with one interface and with two symmetric interfaces, respectively.
	\eprop
	
	\begin{proof}
Clearly, if $u$ is an axisymmetric critical point of $E_\gamma$ with only one interface then the criticality condition \eqref{e:firstvar1D} is trivially satisfied for all $\gamma>0$. To investigate the criticality of axisymmetric patterns with more than one interface, we will need to calculate $v(z_k)$ explicitly. Note that, since
	\[
		v(z_{k+1})=\int_{-1}^{z_{k+1}} \frac{\xi(z)}{1-z^2}\,\exd z,
	\]
we will need to evaluate integrals of $\frac{\xi(z)}{1-z^2}$ on the intervals $(z_k,z_{k+1})$.

Using the notation of Proposition \ref{p:energy1D}, one calculates
\beqn
	\begin{aligned}
		\int_{z_k}^{z_{k+1}}\frac{\xi_k+a_{k+1}(z-z_k)}{1-z^2}\,\exd z &= \int_{z_k}^{z_{k+1}}\frac{\xi_k+a_{k+1}(1-z_k)}{2(1-z)}\,\exd z \\
		&\qquad\qquad+ \int_{z_k}^{z_{k+1}}\frac{\xi_k-a_{k+1}(1+z_k)}{2(1+z)}\,\exd z \\
		&=\frac{\xi_k+a_{k+1}(1-z_k)}{2}\log\frac{1-z_k}{1-z_{k+1}} \\
		&\qquad\qquad+ \frac{\xi_k-a_{k+1}(1+z_k)}{2}\log\frac{1+z_{k+1}}{1+z_k}.
	\end{aligned}
	\nonumber
\eeqn
and, since 
\beqn
 v(z_{k+1})=\sum_{i=0}^{k}\,\int_{z_i}^{z_{i+1}}\frac{\xi(z)}{1-z^2}\,\exd z,
 \nonumber
\eeqn
we get that
\beqn
	v(z_{k+1})-v(z_{k})=\frac{\xi_k+a_{k+1}(1-z_k)}{2}\log\frac{1-z_k}{1-z_{k+1}} \\
		+\frac{\xi_k-a_{k+1}(1+z_k)}{2}\log\frac{1+z_{k+1}}{1+z_k}.
		\label{vdiff}
\eeqn

Now for an axisymmetric critical point $u$ with a partition $-1=z_0<z_1<\ldots<z_{n+1}=1$ determining the interfaces, looking at the criticality condition \eqref{e:firstvar1D} at two consecutive interfaces we see that the equation
	\beqn
		\left(\kappa_g(z_{k+1})-\kappa_g(z_k)\right)+4\gamma\,\left(v(z_{k+1})-v(z_{k})\right)=0
		\label{critdiff}
	\eeqn
is satisfied for an interval of $\gamma$-values if, and only if, both $\kappa_g(z_{k+1})-\kappa_g(z_k)$ and $v(z_{k+1})-v(z_{k})$ vanish simultaneously. But, by \eqref{e:geod-curv}, $\kappa_g(z_{k+1})-\kappa_g(z_k)=0$ implies that $z_{k+1}=-z_k$. Then, \eqref{vdiff} becomes
	\beqn
		v(-z_k)-v(z_k)=(\xi_k-a_{k+1}z_k)\log\frac{1-z_k}{1+z_k}.
		\label{vdiffsymm}
	\eeqn
For $z\in(z_k,-z_k)$, we have that $\xi(z)=a_{k+1}z$; hence, in particular, $\xi(z_k)=\xi_k=a_{k+1}z_k$. Plugging this in \eqref{vdiffsymm} yields that
	\beqn
		v(-z_k)-v(z_k)=0;
		\nonumber
	\eeqn
hence, satisfying the condition \eqref{critdiff} for all $\gamma>0$. This implies that for an axisymmetric critical pattern to be a critical point for an interval of $\gamma$-values consecutive interfaces have to be symmetric with respect to the equator on the sphere. However, this is only possible if the pattern has two interfaces, that is, it is a double cap.
	\end{proof}

\begin{remark}
 We note that both propositions above point out a major difference between the  one dimensional patterns on the two-sphere and the $n$-dimensional flat torus $\domainn$. In the latter case, a lamellar pattern with any number of interfaces is a critical point of the energy $E_\gamma$ for all $\gamma>0$ whereas the rigidity of the criticality condition induced by the curvature of the domain does not allow this to happen either in terms of the number of interfaces one can fit (Proposition \ref{p:numberofinterfaces}) or in terms of the $\gamma$-values for which a pattern of certain number of interfaces stays as a critical point (Proposition \ref{p:intervalofgamma}).
\end{remark}

\bigskip
Proposition \ref{p:intervalofgamma} shows that for any axisymmetric critical point with more than two interfaces the location of those interfaces changes with $\gamma$. For any fixed $\gamma>0$ they are in fact determined by the system of equations \eqref{e:firstvar1D}. In the proceeding discussion we will look at this system a little more closely, and demonstrate the difficulty of solving this finite dimensional system of equations by providing explicit calculations in the case of three and four interfaces. These calculations will also show that the uniformly distributed patterns with 3 and 4 interfaces where the locations of the interfaces are given by
	\beqn
	\left\{-\frac{1}{2},0,\frac{1}{2}\right\}\qquad\mbox{and}\qquad\left\{-\frac{3}{4},-\frac{1}{4},\frac{1}{4},\frac{3}{4}\right\}
	\label{e:3-4-interface-unif}
	\eeqn
are critical points for $\gamma$-values
	\beqn
		\frac{-1}{2\sqrt{3}\log(3/4)}\qquad\mbox{and}\qquad\frac{3/\sqrt{7}+1/\sqrt{15}}{3\log(5/7)+\log3},
		\label{e:3-4-interface-gamma}
	\eeqn
respectively. As we will show in Proposition \ref{p:uniforminterface}, these uniformly distributed patterns are indeed special since, along with the single and double cap, they are the only critical points of $E_\gamma$ with uniformly distributed interfaces.

Given a fixed $\gamma>0$, to determine the exact location of the interfaces $z_k$ of an axisymmetric critical point one needs to solve the highly nonlinear but finite system of equations given by \eqref{e:firstvar1D}. Letting $-1=z_0<z_1<\ldots<z_{n}<z_{n+1}=1$ denote the interfaces, we first note that, for any $z_i$, $z_j$, and $z_k$, the set of equations
  \beqn
		\begin{aligned}
			(\kp(z_i)-\kp(z_j))+4\gamma\,(v(z_i)-v(z_j))&=0\\
			(\kp(z_j)-\kp(z_k))+4\gamma\,(v(z_j)-v(z_k))&=0
		\end{aligned}
	\nonumber
	\eeqn
implies that
	\[
	 (\kp(z_i)-\kp(z_k))+4\gamma\,(v(z_i)-v(z_k))=0.
	\]
Hence, the system \eqref{e:firstvar1D} reduces to
\beqn
	\begin{array}{ccccc}
		(\kp(z_n)-\kp(z_{n-1}))&+&4\gamma\,(v(z_n)-v(z_{n-1}))&=&0\\
		(\kp(z_{n-1})-\kp(z_{n-2}))&+&4\gamma\,(v(z_{n-1})-v(z_{n-2}))&=&0\\
		\vdots&+&\vdots &=&\vdots\\
		(\kp(z_2)-\kp(z_{1}))&+&4\gamma\,(v(z_2)-v(z_{1}))&=&0
	\end{array}
	\label{generalsystem}
\eeqn
along with the mass constraint
 \beqn
 	-(1-z_n)-(z_1+1)+\sum_{i=2}^n (-1)^i(z_i-z_{i-1})=0.
 	\label{e:massconstninterfaces}
 \eeqn
 
Taking the sign of $\kappa_g(z_k)$ into account, and recalling \eqref{vdiff}, the system \eqref{generalsystem} becomes
\makeatletter
    \def\tagform@#1{\maketag@@@{\normalsize(#1)\@@italiccorr}}
\makeatother
{\tiny
	\beqn
		\begin{array}{ccccc}
		\left(-\frac{z_n}{\sqrt{1-z_n^2}}-\frac{z_{n-1}}{\sqrt{1-z_{n-1}^2}}\right)\!\!&+&\!\!4\gamma\,\left(\frac{\xi_{n-1}+a_{n}(1-z_{n-1})}{2}\log\frac{1-z_{n-1}}{1-z_{n}}
		+\frac{\xi_{n-1}-a_{n}(1+z_{n-1})}{2}\log\frac{1+z_{n}}{1+z_{n-1}}\right)\!\!&=&\!\!0\\
		\left(\frac{z_{n-1}}{\sqrt{1-z_{n-1}^2}}+\frac{z_{n-2}}{\sqrt{1-z_{n-2}^2}}\right)\!\!&+&\!\!4\gamma\,\left(\frac{\xi_{n-2}+a_{n-1}(1-z_{n-2})}{2}\log\frac{1-z_{n-2}}{1-z_{n-1}}
		+\frac{\xi_{n-2}-a_{n-1}(1+z_{n-2})}{2}\log\frac{1+z_{n-1}}{1+z_{n-2}}\right)\!\!&=&\!\!0\\
		\vdots&+&\vdots &=&\vdots\\
		\left(-\frac{z_2}{\sqrt{1-z_2^2}}-\frac{z_{1}}{\sqrt{1-z_{1}^2}}\right)\!\!&+&\!\!4\gamma\,\left(\frac{\xi_{1}+a_{2}(1-z_{1})}{2}\log\frac{1-z_{1}}{1-z_{2}}
		+\frac{\xi_{1}-a_{2}(1+z_{1})}{2}\log\frac{1+z_{2}}{1+z_{1}}\right)\!\!&=&\!\!0,
	\end{array}	
	\label{e:fin-syst-crit}
	\eeqn
}
where $\xi_i$ is defined in \eqref{xik}, and $a_i=(-1)^i-m$. Together with \eqref{e:massconstninterfaces} this nonlinear system contains $n$ equations involving the variables $z_1,\ldots,z_n$.

Next we will show how we can find patterns with 3 or 4 interfaces by using explicit calculations. These calculations also show that for these patterns to be critical points the parameter $\gamma$ needs to be sufficiently large.

Let us start with the pattern with 3 interfaces. Suppose that the interfaces are located at
	\[
		\{-z_1,0,z_1\}
	\]
for some $0<z_1<1$. Note that the geodesic curvatures at these interfaces are
	\[
		\kappa_g(-z_1)=\frac{-z_1}{\sqrt{1-z_1^2}},\quad \kappa_g(0)=0,\mbox{ and}\quad \kappa_g(z_1)=\frac{z_1}{\sqrt{1-z_1^2}}
	\]
taking into account that $u=-1$ on $[0,-z_1]$.

Then, if this pattern is a critical point for a $\gamma$-value, it will satisfy the criticality condition
	\[
		(\kappa_g(z_1)-\kappa_g(0))+4\gamma(v(z_1)-v(0))=0;
	\]
hence, using \eqref{vdiff}, we get that
	\[
		\frac{z_1}{\sqrt{1-z_1^2}} + 4\gamma[z_1\log(1+z_1)-(z_1-1)\log(1-z_1)]=0.
	\]
Solving for $\gamma$, then, yields
  \beqn
  	\gamma=\frac{-z_1/\sqrt{1-z_1^2}}{4[z_1\log(1+z_1)-(z_1-1)\log(1-z_1)]}.
  	\label{gamma3interfaces}
  \eeqn
It is easy to see that as $z_1$ approaches $0^{+}$ (i.e., as $-z_1\arrow0^{-}$), $\gamma$ approaches 1/4. Also, the curve defined above has vertical asymptotes at $z_1\approx\pm 0.69$ (see the Figure \ref{threeinterfaces}), and we see that for any value of $\gamma>1/4$ there exists a location for the first interface $-1<-z_1<0$ such that the pattern with 3 interfaces is a critical point.

\begin{figure}[ht!]
\begin{center} 
  		\subfigure[{\tiny $\gamma$ in the $(-z_1,\gamma)$-plane for 3 interfaces}]{
            \label{threeinterfaces}
            \includegraphics[width=0.35\linewidth]{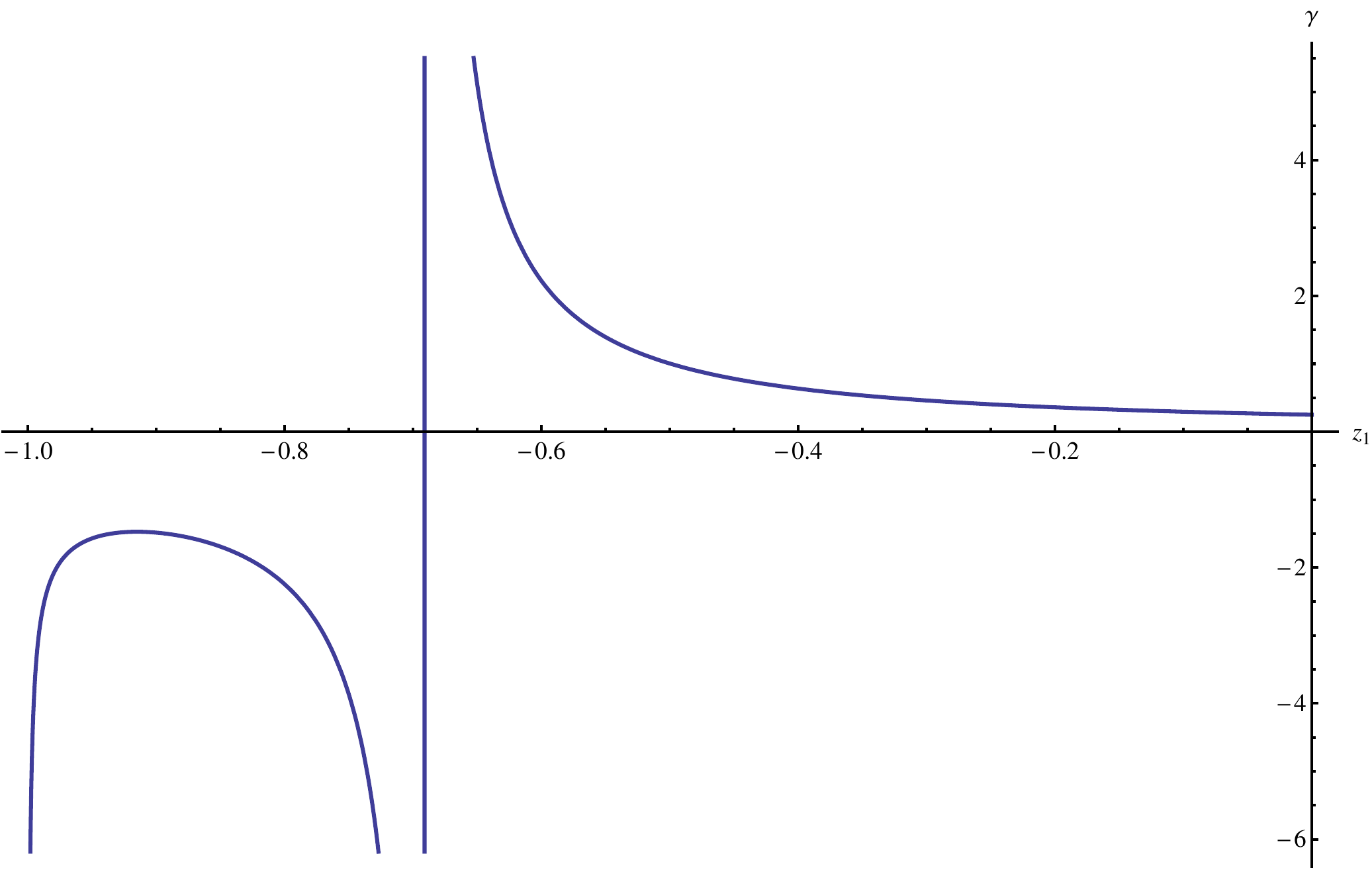}
        }\qquad
        \subfigure[{\tiny $\gamma$ in the $(-z_1,\gamma)$-plane for 4 interfaces}]{
           \label{fourinterfaces}
           \includegraphics[width=0.35\linewidth]{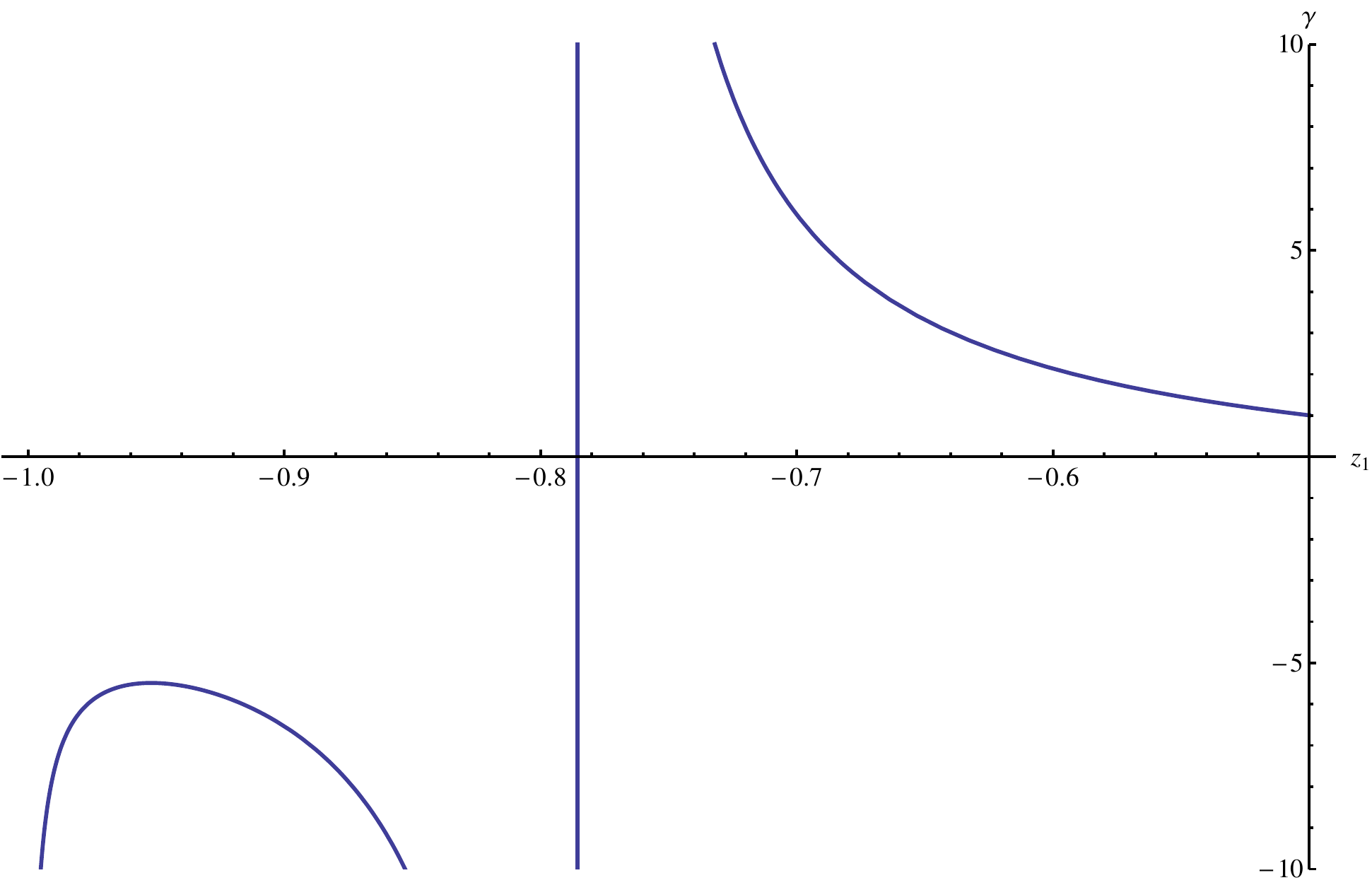}
        }
   \caption{The curve $\gamma$ in the $(-z_1,\gamma)$-plane for 3 and 4 interfaces, respectively}
\end{center} 
\end{figure}  

For the pattern with 4 symmetric interfaces let the interfaces be located at
	\[
		\left\{-z_1,\frac{1}{2}-z_1,z_1-\frac{1}{2},z_1\right\}
	\]
for some $1/2<z_1<1$.
 
It is easy to check that due to the symmetry assumption the difference between the criticality conditions on any two interfaces is equivalent to the difference between two consecutive interfaces, that is, if this pattern is a critical point for some $\gamma>0$ it satisfies the equation
	\[
		[\kappa_g(z_1)-\kappa_g(z_1-1/2)]+4\gamma[v(z_1)-v(z_1-1/2)]=0.
	\]
Again, using \eqref{vdiff} for the second term, and solving for $\gamma$ yields
	\beqn
		\gamma=\frac{\frac{z_1}{\sqrt{1-z_1^2}}+\frac{z_1-1/2}{\sqrt{1-(z_1-1/2)^2}}}{4\left(-z_1\log\left(\frac{1+z_1}{1/2+z_1}\right)-(z_1-1)\log\left(\frac{3/2-z_1}{1-z_1}\right)\right)}.
		\label{gamma4interfaces}
	\eeqn

This curve has two asymptotes at $z_1\approx\pm 0.78554$. This shows that for any $\gamma>0$, the first and the last interfaces are away from the poles. Moreover, $\gamma\arrow\frac{1}{2\sqrt{3}\log(4/3)}\approx1.00345$ (see Figure \ref{fourinterfaces}) as $-z_1\arrow-1/2^{-}$ or $z_1\arrow1/2^{+}$. Therefore for any $\gamma>\frac{1}{2\sqrt{3}\log(4/3)}$ there exists a critical point with 4 interfaces.

Looking at these calculations, it is then easy to determine the exact $\gamma$-values given in \eqref{e:3-4-interface-gamma} for which the uniformly distributed axisymmetric patterns with 3 and 4 interfaces given by \eqref{e:3-4-interface-unif} are critical points.
 
\bigskip

Next we show that the patterns with one, two, three or four interfaces are the only ones with uniform interface distribution.

	\bprop[Uniform interface distribution]\label{p:uniforminterface}
		Any uniformly distributed pattern with the number of interfaces greater than 4 is not a critical point.
	\eprop 
	
	\begin{proof}
Clearly the single and double cap, and the patterns given by \eqref{e:3-4-interface-unif} are uniformly distributed axisymmetric critical points of $E_\gamma$. Suppose $u$ is a uniformly distributed axisymmetric pattern with more than 4 interfaces. First, suppose that $u(z)$ has $2n-1$ interfaces with $n\geq3$. Then there are $2n$ regions where $u=1$ and $2n$ regions where $u=-1$ each having a height of $1/n$. That is, the locations of the interfaces are given by $z_i=-1+i/n$ in $z$-coordinates. Let us concentrate on the interval $(z_2,z_3)$ on which $u=-1$. Note that $z_3=-1+3/n\leq0$ since $n\geq3$. By \eqref{critdiff}, we have that
	\beqn
	 \left(\kappa_g(z_3)-\kappa_g(z_2)\right)+4\gamma\,\left(v(z_3)-v(z_2)\right)=0.
	 \label{critdiffodd}
	\eeqn
As $u=-1$ on $(z_2,z_3)$, the unit normals to the interfaces point inwards; hence,
	\[
		\kappa_g(z_3)-\kappa_g(z_2) < 0.
	\]

On the other hand, since $z_1-(-1)=z_2-z_1$ because of the uniform area distribution, the function $\xi(z)$ is negative on the interval $(z_2,z_3)$. Thus,
	\[
		v(z_3)-v(z_2)=\int_{z_2}^{z_3} \frac{\xi(z)}{1-z^2}\,\exd z < 0,
	\]
which yields that the equation \eqref{critdiffodd} cannot be satisfied for any $\gamma>0$.

Next, suppose that $u(z)$ has $2n$ interfaces with $n\geq3$; that is, there are $n$ regions where $u=1$ and $n+1$ regions where $u=-1$. For equal area distribution, we choose the interfaces such that the interval $(-1,z_1)$ and $(z_{2n},1)$ have length $\frac{1}{2n}$, and the interval $(z_i,z_{i+1})$ have length $\frac{1}{n}$ for $i=1,\ldots,2n-1$. Therefore the locations of the interfaces are given by the formula
	\[
		z_i=-1+\frac{2i-1}{2n}
	\]
for $i=1,\ldots,2n$.

Writing the criticality conditions between the interfaces $z_1$ and $z_2$, and between $z_2$ and $z_3$ as in \eqref{critdiff}, and assuming that there is a $\gamma$-value satisfying both equations, we obtain 
	\[		\frac{\kappa_g(z_1)-\kappa_g(z_2)}{v(z_2)-v(z_1)}=\frac{\kappa_g(z_2)-\kappa_g(z_3)}{v(z_3)-v(z_2)}.
	\] 
For the choices of $z_1$, $z_2$, and $z_3$ given above for equal area distribution, this equation, then, becomes
	\beqn		\frac{\frac{\frac{1-2n}{2n}}{\sqrt{1-\left(\frac{1-2n}{2n}\right)^2}}+\frac{\frac{3-2n}{2n}}{\sqrt{1-\left(\frac{3-2n}{2n}\right)^2}}}{\left(\frac{2n-1}{2n}\right)\log\left(\frac{4n-1}{4n-3}\right)-\frac{1}{2n}\log3}=\frac{-\frac{\frac{3-2n}{2n}}{\sqrt{1-\left(\frac{3-2n}{2n}\right)^2}}-\frac{\frac{5-2n}{2n}}{\sqrt{1-\left(\frac{5-2n}{2n}\right)^2}}}{\left(\frac{1-n}{n}\right)\log\left(\frac{4n-3}{4n-5}\right)-\frac{1}{n}\log\left(\frac{5}{3}\right)}
		\nonumber
	\eeqn
which has no solution for $n\geq3$. Therefore any axisymmetric pattern with more than four uniformly distributed interfaces is not a critical point of $E_\gamma$.
	\end{proof}

Finally, we will make use of the rigidity of the criticality condition to establish a lower bound on the polar cap diameter of an axisymmetric critical point when $m=0$.

 \bprop\label{p:lowerbound}
 	Let $u$ be an axisymmetric critical point of $E_\gamma$ with $n\geq 2$ interfaces and let $z_1$ denote the first interface. Then the diameter of the polar cap determined by $z_1$ is bounded from below. More specifically,
 \[
	z_1 \geq \frac{a}{\sqrt{1+a^2}}
 \]
for $a:=-\frac{6\gamma}{e} - \frac1{\sqrt3}$. 
	\eprop
	
	\begin{proof}
Let $n\geq2$ and suppose that $z_1\leq-\frac12$, which implies that $z_2\leq\frac12$ by mass constraint. Then the criticality condition gives
	\beqn
		\frac{|z_1|}{\sqrt{1-z_1^2}} \leq 4\gamma |v(z_2)-v(z_1)| + \frac{|z_2|}{\sqrt{1-z_2^2}} \leq 			4\gamma |v(z_2)-v(z_1)| + \frac{1}{\sqrt{3}}.
		\nonumber
	\eeqn
Recall that since $\|u\|_{L^\infty}\leq1$, we have $v\in C^{1,\alpha}(\domainS)$ for any $\alpha\in(0,1)$, with $|\nabla v|$ uniformly bounded independent of $u$.
Hence there are constants $C_1$ and $C_2$ such that
		\beqn
				\frac{z_1^2}{1-z_1^2} \leq C_1\gamma^2+ C_2,
		\eeqn
leading to
		\beqn
				|z_1|^2 \leq \frac{C_1\gamma^2 + C_2}{C_1\gamma^2+C_2+1} = 1 - \frac1{C_1\gamma^2+C_2+1}.
		\nonumber
		\eeqn
This means that $\eps:=z_1+1$ is bounded from below by a constant multiple of $1/\gamma^2$,
or in other words, that the diameter $\ell$ of the polar cap is bounded below by $\sim1/{\gamma}$.

Now we obtain the more quantitative bound. We have
		\beqn\label{e:zz}
				\frac{z_1}{\sqrt{1-z_1^2}} + \frac{z_2}{\sqrt{1-z_2^2}} = 4\gamma (v(z_2)-v(z_1)).
		\eeqn
On the other hand, \eqref{vdiff} gives
		\beqn
				v(z_{2})-v(z_{1}) = -z_1\log\frac{1-z_1}{1-z_{2}} - (1+z_1)\log\frac{1+z_{2}}{1+z_1} > -(1+z_1)\log\frac{1+z_{2}}{1+z_1},
		\nonumber
		\eeqn
where we have taken into account that $\xi_1 = - (z_1+1) = - \eps$ and that $z_1<z_2$.
Then assuming that $z_2\leq\frac12$ as before, we can derive the bound
		\beqn
				(1+z_1)\log\frac{1+z_{2}}{1+z_1} \leq \frac3{2e}.
 		\nonumber
		\eeqn
Using this in \eqref{e:zz}, we get that
		\beqn
			\frac{z_1}{\sqrt{1-z_1^2}} \geq -\frac{6\gamma}{e} - \frac1{\sqrt3},
 		\nonumber
		\eeqn
and this, in turn, yields that
		\[
			z_1 \geq \frac{a}{\sqrt{1+a^2}}
		\]
with $a=-\frac{6\gamma}{e} - \frac1{\sqrt3}$ as stated.
	\end{proof}

\begin{remark}\label{r:lowerbound}
Similar arguments can be applied to obtain a lower bound on the interfacial distance. Namely, for any fixed $\gamma>0$, we have that
	\[
		|z_{k+1}-z_k|\geq \frac{C\max\{|\kappa_g(z_k)|,|\kappa_g(z_{k+1})|\}}{\gamma} \geq \frac{C\max\{|z_k|,|z_{k+1}|\}}{\gamma}
	\]
for some constant $C>0$.
\end{remark}

\section{Remarks on Instability}\label{sec:instability}

In this short section, we give a few remarks concerning stability. Instability of the double cap was obtained in \cite{Top} by looking at  the second variation of the energy $E_\gamma$, and it is natural to explore a similar analysis for general critical points. Unfortunately, we do not know the \emph{exact} location of the interfaces for general critical points and as we explain, this presents many difficulties. For example, consider the following two standard techniques of obtaining instability of a critical point.

	\begin{itemize}
		\item (Fluctuations of a boundary component) Suppose we have an equatorially symmetric, axisymmetric critical point with $n$ interfaces given by
	\[
		-1<z_1<z_2<\ldots<z_n<1.
	\]
Then, by Proposition \ref{p:lowerbound}, we have that
	\beqn
		1/2 < z_n < C(\gamma) \leq 1
		\label{upperboundzn}
	\eeqn
where $C(\gamma)$ is given by
	\[
		C(\gamma)=\frac{\frac{6\gamma}{e}+\frac{1}{\sqrt{3}}}{\sqrt{1+\left(\frac{6\gamma}{e}+\frac{1}{\sqrt{3}}\right)^2}}.
	\]

Take $f$ defined on $\pt A = \Gamma_1 \cup \cdots \cup \Gamma_n$ using spherical coordinates by
	\[
		f_k(\theta) = \begin{cases} \sin(k\theta) &\mbox{ on } \Gamma_n \\
0 & \mbox{ otherwise} \end{cases}
	\]
for some integer $k$.

Using \eqref{upperboundzn}, the second variation \eqref{secondvar} expressed in $z$-coordinates becomes
	\beqn
	 \begin{aligned}
	 \frac{1}{\pi} J(f_k) &= \frac{k^2-1}{\sqrt{1-z_n^2}} + 4\gamma\left(\frac{1-z_n^2}{k}+(z_n-1)\right) \\
	 										  &\leq \frac{k^2-1}{\sqrt{1-C^2(\gamma)}}+\frac{3\gamma}{k}+4\gamma(C(\gamma)-1).
	 	\label{singlesinpert}
	 \end{aligned}
	\eeqn
Written out explicitly, this gives us
 \[
 \frac{1}{\pi}J(f_k)\leq	  (k^2-1)\sqrt{1+\left(\frac{6\gamma}{e}+\frac{1}{\sqrt{3}}\right)^2}+\frac{3\gamma}{k}+\frac{4\gamma\left(\left(\frac{6\gamma}{e}+\frac{1}{\sqrt{3}}\right)-\sqrt{1+\left(\frac{6\gamma}{e}+\frac{1}{\sqrt{3}}\right)^2}\right)}{\sqrt{1+\left(\frac{6\gamma}{e}+\frac{1}{\sqrt{3}}\right)^2}}.
 \]
 
 Note that the only negative term above is the third term. For large $\gamma$-values the third term is close to zero; however, since small circles are very stable critical points of the perimeter functional, the first term is {\em very} positive, making it difficult to establish instability. This pathology can also be seen in the first line of \eqref{singlesinpert} where for fixed $\gamma$ there is only one free parameter, namely $k$; however, we lack information on the exact location (hence, it's closeness to 1) of the last interface $z_n$.\\
 
  \item (Axisymmetric perturbations) Another easy trick to establish instability is to consider an initial perturbation $f$ of the form
  \[
		f(x) = \begin{cases} -1 &\mbox{ on } \Gamma_n, \\
		                      1 &\mbox{ on } \Gamma_1, \\
                          0 & \mbox{ otherwise}. \end{cases}
	\]

Then, using the fact that
	\[
		\log(|x-y|)<\log2
	\]
for $x\in\Gamma_1$ and $y\in\Gamma_n$, the second variation in $z$-coordinates is given by
	\[
		\frac{1}{\pi}J(f) \leq \frac{-4}{\sqrt{1-z_n^2}}+\gamma\left(32(1-z_n^2)\left(\log2-\log(\sqrt{1-z_n^2})\right)\right)+\gamma(16(z_n-1)).
	\]
		The above formula shows that we get instability for fixed $\gamma$ if $z_n$ is close enough to 1. We expect that this will be the case if the number of interfaces $n$ is sufficiently large.\\
	\end{itemize}
	
 The difficulty of proving stability of an axisymmetric critical pattern lies in the structure of $u$. There one needs to consider the interaction between components of $\pt A$ given a general smooth function $f$ on $\pt A$. Even for the simplest axisymmetric critical point, namely the symmetric double cap, one encounters integrals of the form
 	\beqn
 		\begin{split}
			\int_0^{2\pi}\!\!\!\int_0^{2\pi} \log(5-3\cos(\theta-\alpha))&\cos(n\theta)\cos(n\alpha)\,d\theta d\alpha \\
		 &=\int_0^{2\pi}\!\!\!\int_0^{2\pi} \log(5-3\cos(\theta-\alpha))\sin(n\theta)\sin(n\alpha)\,d\theta d\alpha.
		\end{split}
		\label{e:doublecapstab}
	\eeqn 
Relying on numerical computations we claim that the integrals in \eqref{e:doublecapstab} are given by the closed formula
	\[
		-\frac{2\pi^2}{n3^n};
	\]
however, at this stage, we cannot prove this analytically.\\

\medskip
Instability towards spiraling is particularly interesting and should be expected. Indeed 
Figure \ref{fig:spiraling} shows a hybrid numerical simulation of a low energy state for (\ref{nlpch}) on $\domainS$. 
Starting from random initial conditions, one runs the $H^{-1}$ gradient flow of  (\ref{nlpch})  until a metastable pattern is reached. One then runs motion by mean curvature for a fixed number of time steps followed again by the gradient flow. This is repeated several times. The motion by mean curvature flow, which in general increases the overall energy, is there to surmount energy barriers and to allow us to access as low an energy configuration as possible. The final spiral pattern in  Figure \ref{fig:spiraling} is the result of this algorithm.

  \begin{figure}[ht!]
    \includegraphics[width=0.4\textwidth]{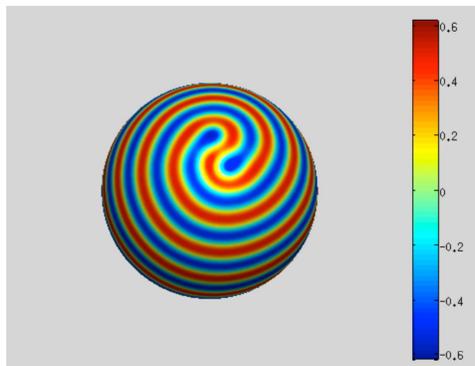}
  \caption{Result of a hybrid numerical simulation of  B. Shahriari, S. Ruuth, and R. Choksi. 
  The result shows a low energy stable state.}
  \label{fig:spiraling}
 	\end{figure}

\medskip

\noindent {\bf Acknowledgements.} I.T. would like to thank Peter Sternberg for discussions regarding the stability of the double cap critical point. This research was supported by NSERC Canada Discovery Grants. I.T. was also partially supported by the Applied Mathematics Laboratory of the Centre de Recherches Math\'{e}matiques.


\end{document}